\documentclass[12pt]{article}
\pdfoutput=1
\usepackage{fullpage}

\usepackage{graphicx}
\usepackage{tikz}
\usepackage{caption}
\usepackage{hyperref}
\usepackage{subcaption}
\usepackage[T1]{fontenc}
\usepackage{tikz}
\usetikzlibrary{positioning}
\usepackage{pgfplots}
\pgfplotsset{compat=1.5.1}
\usepgfplotslibrary{dateplot}
\usepgfplotslibrary{fillbetween}
\usetikzlibrary{patterns}
\usepackage{color}
\colorlet{mygreen}{black!50!green!60}
\colorlet{myblue}{black!30!blue}

\newcommand{\open}[1]{]#1[}
\newcommand{\leftopen}[1]{]#1]}
\newcommand{\rightopen}[1]{[#1[}

\newcommand{\lzerovalue}{l}
\newcommand{\support}{L}
\newcommand{\dualsupport}{I}
\newcommand{\OriginalSupportNorm}[2]{{#1}_{(#2)}^{\mathrm{sn}}}
\newtheorem{fact}{Fact}

\usepackage[square,numbers]{natbib}
\usepackage{amsfonts}
\usepackage{amsmath}
\usepackage{amssymb}
\usepackage{fullpage}
\graphicspath{{Figures/}}
\def\keywordsname{Keywords}
\def\mykeywords{\par\addvspace{17pt}\small\rmfamily
  \trivlist\if!\keywordsname!\item[]\else\item[\hskip\labelsep
  {\bfseries\keywordsname}]\fi}

\newenvironment{keywords}{\begin{mykeywords}}{\end{mykeywords}}

\usepackage{stmaryrd} 
\usepackage{wasysym} 
\usepackage[colors]{optsys}
\renewcommand{\primalbis}{\primal'}
\renewcommand{\PRIMAL}{\RR^d}
\renewcommand{\DUAL}{\RR^d}

\newcommand{\Lzero}{\mathcal{L}_0}

\renewcommand{\Capra}{Capra}

\title{The \Capra-subdifferential of the $\lzero$ pseudonorm}

\author{Adrien~Le~Franc\thanks{CERMICS, \'{E}cole des Ponts ParisTech, France.}
  \and
  Jean-Philippe~Chancelier\footnotemark[1]\and
  Michel~De~Lara\footnotemark[1]
}

\begin{document}
\maketitle
\begin{abstract}
  The $\lzero$ pseudonorm counts the nonzero coordinates of a vector.
  It is often used in optimization problems to enforce the sparsity of the solution.
  However, this function is nonconvex and noncontinuous,
  and optimization problems formulated with $\lzero$ --- be it in the objective
  function or in the constraints --- are hard to solve in general.
  Recently, a new family of coupling functions --- called \Capra\ (constant
  along primal rays) --- has proved to induce relevant generalized Fenchel-Moreau conjugacies
  to handle the $\lzero$ pseudonorm.
  In particular, under a suitable choice of source norm on~$\RR^d$ ---
  used in the definition of the \Capra\ coupling --- 
  the function~$\lzero$ has nonempty \Capra-subdifferential, hence is
  \Capra-convex. 		In this article, we give explicit formulations
  for the \Capra-subdifferential of the~$\lzero$ pseudonorm, when the source norm
  is a $\ell_p$ norm with $p \in \nc{1,\infty}$.
  We illustrate our results with graphical visualizations of the \Capra-subdifferential
  of $\lzero$ for the Euclidean source norm.
\end{abstract}


\section{Introduction}
  

The $\lzero$ pseudonorm is a function which counts 
the number of nonzero elements of a vector. 
This function appears in numerous optimization problems 
to enforce the sparsity of the solution.
As this function is nonconvex and noncontinuous, 
the powerful framework of convex analysis is unadapted
to address such problems, unless considering a convex relaxation of the function $\lzero$.
In a recent series of works
\cite{Chancelier-DeLara:2021_ECAPRA_JCA,Chancelier-DeLara:2022_CAPRA_OPTIMIZATION,Chancelier-DeLara:2021_SVVA}, 
it was shown that conjugacies
induced by the so-called \Capra\ (constant along primal rays) coupling
are well-suited to handle the $\lzero$ pseudonorm.
In particular, the authors show in \cite{Chancelier-DeLara:2021_SVVA}
that --- for a large class of source norms 
(that encompasses the $\ell_p$ norms for $p \in \open{1, \infty}$)
employed in the definition of the \Capra\ coupling ---
the $\lzero$ pseudonorm is equal to its \Capra-biconjugate,
meaning that it is a \Capra-convex function.
They also provide formulae for the \Capra-subdifferential
of $\lzero$ in \cite{Chancelier-DeLara:2022_CAPRA_OPTIMIZATION},
and prove that this subdifferential is nonempty for the same class
of source norms that guarantee the \Capra-convexity of $\lzero$
in \cite{Chancelier-DeLara:2021_SVVA}.

The formulation of the \Capra-subdifferential 
of the $\lzero$ pseudonorm in \cite{Chancelier-DeLara:2022_CAPRA_OPTIMIZATION}
involves the so-called coordinate-$k$ and dual coordinate-$k$ norms,
defined by variational expressions,
and is not readily computable.
The main contribution of this article is to derive explicit
formulations to compute the \Capra-subdifferential 
of the $\lzero$ pseudonorm for all $\ell_p$ source norms with $p \in \nc{1, \infty}$.
Subsequently, we comment on the domain of these subdifferentials,
and extend previous results by showing that, in the extreme cases where $p \in \na{1, \infty}$,
the $\lzero$ pseudonorm is not \Capra-convex.
We also provide graphical illustrations of the \Capra-subdifferential
of~$\lzero$, 
and compare it with other notions of generalized subdifferentials
for~$\lzero$ found in \cite{Le:2013}.
With the \Capra-subdifferential, we can naturally derive
``polyhedral-like''~\cite[p.~114]{Singer:1997} lower bounds
for the $\lzero$ pseudonorm, that is, lower bounds that are the maximum of a
finite number of so-called \emph{\Capra-affine} functions. 

The paper is organized as follows.
First, we recall background notions on \Capra-couplings
in \S\ref{sec:background_on_capra}. Second, we derive explicit formulations
for the \Capra-subdifferential of $\lzero$ in \S\ref{sec:capra_subdifferential_l0}.
Finally, we provide illustrative visualizations and discuss 
the positioning of the \Capra-subdifferential of $\lzero$
with respect to other notions of subdifferentials in \S\ref{sec:visualization_and_discussion}.

\section{Background on the \Capra\ coupling and the $\lzero$ pseudonorm}
\label{sec:background_on_capra}

For any pair of integers $i \leq j$, we denote
\( \ic{i,j}=\na{i,i+1, \dots, j-1,j} \). 
We work on the Euclidean space $\PRIMAL$, where $d \in \NN^*$,
equipped with the canonical scalar product $\proscal{\cdot}{\cdot}$,
and with a norm $\TripleNorm{\cdot}$ that we call the \emph{source norm}.
We stress the point that $\TripleNorm{\cdot}$ can be \emph{any} norm,
and is not required to be the Euclidean norm.
We denote the unit sphere and the unit ball 
of the norm~$\TripleNorm{\cdot}$ by, respectively,
\begin{equation}
  \label{eq:triplenorm_unit_sphere_ball}
  \TripleNormSphere
  = 
  \defset{\primal \in \RR^d}{\TripleNorm{\primal} = 1} 
  \quad \text{ and } \quad
  \TripleNormBall 
  = 
  \defset{\primal \in \RR^d}{\TripleNorm{\primal} \leq 1} 
  \eqfinv
\end{equation}
or, more explicitly,	by $\SPHERE_{\TripleNorm{\cdot}}$
and $\BALL_{\TripleNorm{\cdot}}$ when needed.

First, we recall the definition of the so-called \Capra\ coupling
and of the resulting \Capra\ conjugacy in~\S\ref{sec:capra_conjugacies}.
Second, we review the main results relating \Capra\ conjugacies
and the $\lzero$ pseudonorm in~\S\ref{sec:capra_convexity_of_l0}.

\subsection{\Capra\  conjugacies}
\label{sec:capra_conjugacies}

We start by recalling the definition of the \Capra\ coupling.

\begin{definition}[\cite{Chancelier-DeLara:2022_CAPRA_OPTIMIZATION}, Definition~4.1]
  Let $\TripleNorm{\cdot}$ be a norm on~$\RR^d$.
  We define the \emph{coupling}~$\CouplingCapra : \PRIMAL \times \DUAL \to \RR$
  between $\RR^d$ and $\RR^d$,
  that we call the \Capra~coupling, by
  \begin{equation}
    \forall \dual \in \RR^d \eqsepv 
    \CouplingCapra\np{\primal, \dual} =
    \begin{cases}
      \frac{ \proscal{\primal}{\dual} }{ \TripleNorm{\primal} }
      \eqsepv &\text{ if } \primal \neq 0 \eqfinv
      \\
      0 \eqsepv &\text{ if } \primal = 0 \eqfinp
    \end{cases}
    \label{eq:coupling_CAPRA}
  \end{equation}
  \label{de:Capra}%
\end{definition}

A coupling function such as the \Capra~coupling $\CouplingCapra$ given in Definition~\ref{de:Capra}
gives rise to generalized Fenchel-Moreau conjugacies \cite{Singer:1997, Martinez-Legaz:2005},
that we briefly recall.
Let us introduce the extended real line $\barRR = \RR \cup \na{+\infty, -\infty}$ 
and consider a function $\fonctionprimal : \PRIMAL \to \barRR$.
The $\CouplingCapra$-Fenchel-Moreau conjugate of $\fonctionprimal$
is the function $\SFM{ \fonctionprimal }{\CouplingCapra} : \DUAL \to \barRR$
defined by
\begin{subequations}
  \label{eq:Capra_conjugacies}%
  \begin{equation}
    \SFM{ \fonctionprimal }{\CouplingCapra}(\dual)
    = 
    \sup_{\primal \in \PRIMAL} \Bp{ \CouplingCapra\np{\primal,\dual} 
      -\fonctionprimal\np{\primal}  } 
    \eqsepv \forall \dual \in \DUAL
    \eqfinv
    \label{eq:Fenchel-Moreau_conjugate}
  \end{equation}
  and the $\CouplingCapra$-Fenchel-Moreau biconjugate of $\fonctionprimal$
  is the function $\SFMbi{ \fonctionprimal }{\CouplingCapra} : \PRIMAL \to \barRR$
  defined by
  \begin{equation}
    \SFMbi{ \fonctionprimal }{\CouplingCapra}(\primal)
    = 
    \sup_{\dual \in \DUAL} \Bp{ \CouplingCapra\np{\primal,\dual} 
      -\SFM{ \fonctionprimal }{\CouplingCapra}(\dual)  } 
    \eqsepv \forall \primal \in \PRIMAL
    \eqfinp
    \label{eq:Fenchel-Moreau_biconjugate}
  \end{equation}
  Moreover, we have the inequality
  \begin{equation}
    \SFMbi{ \fonctionprimal }{\CouplingCapra}(\primal)
    \leq \fonctionprimal(\primal)
    \eqsepv \forall \primal \in \PRIMAL \eqfinv
    \label{eq:biconjugate_inequality}
  \end{equation}
\end{subequations}
and following \cite{Martinez-Legaz:2005}, we say that the function $\fonctionprimal$
is \Capra-convex \IFF we have an equality in~\eqref{eq:biconjugate_inequality}.
Lastly, \Capra\ conjugacies also induce a notion of \Capra-subdifferential.
The \Capra-subdifferential of $\fonctionprimal$ is the
set-valued mapping
$\partial_{\CouplingCapra} \fonctionprimal : \PRIMAL \rightrightarrows \DUAL$ defined by
\begin{subequations}
  \label{eq:Capra_subdifferential_notions}%
  \begin{equation}
    \dual \in \partial_{\CouplingCapra}\fonctionprimal\np{\primal} 
    \iff
    \SFM{\fonctionprimal}{\CouplingCapra}\np{\dual} =
    \CouplingCapra\np{\primal,\dual} -\fonctionprimal\np{\primal} 
    \eqfinv
    \label{eq:Capra_subdifferential}%
  \end{equation}
  and which takes closed and convex set values~\cite[Proposition~1]{Chancelier-DeLara:2021_SVVA}.
  We say that \cite[Definition~10.1]{Singer:1997}
  the function $\fonctionprimal$ is \emph{\Capra-subdifferentiable
  at $\primal \in \PRIMAL$} when $\partial_{\CouplingCapra}\fonctionprimal\np{\primal} \neq \emptyset$,
  and we introduce the domain of 
  $\partial_{\CouplingCapra} \fonctionprimal$,
  defined as the set
  \begin{equation}
    \dom\bp{\partial_{\CouplingCapra} \fonctionprimal} = 
    \defset{\primal \in \PRIMAL}{\partial_{\CouplingCapra} \fonctionprimal(\primal) \neq \emptyset}
    \eqfinp
  \end{equation}
\end{subequations}

Observe that if we replace the \Capra\ coupling $\CouplingCapra$
with the scalar product $\proscal{\cdot}{\cdot}$ 
in \eqref{eq:Capra_conjugacies} and \eqref{eq:Capra_subdifferential_notions},
we retrieve well-known notions of standard convex analysis.
We refer to \cite{Chancelier-DeLara:2022_CAPRA_OPTIMIZATION}
for a more complete introduction to \Capra\ conjugacies.

\subsection{\Capra-convexity and \Capra-subdifferentiability of the $\lzero$ pseu\-do\-norm}
\label{sec:capra_convexity_of_l0}

We define the \emph{support} of a vector $\primal \in \PRIMAL$ by
$\SupportMapping(\primal) = \bset{ j \in \na{1,\ldots,d} }{\primal_j \not= 0 }$.
The \emph{\lzeropseudonorm} is the function
\( \lzero : \RR^d \to \na{0,1,\ldots,d} \)
defined~by 
\begin{equation}
  \lzero\np{\primal} =
  \bcardinal{ \SupportMapping(\primal) }
  \eqsepv \forall \primal \in \RR^d
  \eqfinv
  \label{eq:pseudo_norm_l0}  
\end{equation}
where $\cardinal{\IndexSubset}$ denotes the cardinality of 
a subset \( \IndexSubset \subseteq \na{1,\ldots,d} \).
We recall the main results relating
the \Capra\ coupling $\CouplingCapra$ of Definition~\ref{de:Capra} and the $\lzero$ pseudonorm.
To ease the reading, we gather the required background notions on norms in Appendix~\ref{app:background_on_norms}.

First, we recall that, under a suitable choice of source norm,
the $\lzero$ pseudonorm is \Capra-subdifferentiable everywhere on $\RR^d$,
hence is a \Capra-convex function.
\begin{theorem}[from \cite{Chancelier-DeLara:2021_SVVA}, Theorem~1 and Proposition~2]
  \label{th:capra_convexity}%
  Let $\TripleNorm{\cdot}$ be the source norm
  employed for the \Capra\ coupling $\CouplingCapra$
  in Definition~\ref{de:Capra}.
  If both the norm $\TripleNorm{\cdot}$ and the dual norm $\TripleNormDual{\cdot}$
  are orthant-strictly monotonic (see Definition~\ref{de:orthant-monotonic}),
  then we have that
  \begin{equation*}
    \partial_{\CouplingCapra} \lzero(\primal) 
    \neq \emptyset \eqsepv \forall \primal \in \RR^d
    \eqfinp
  \end{equation*}
  As a consequence, we have that         
  \begin{equation*}
    \SFMbi{ \lzero }{\CouplingCapra} 
    = \lzero 
    \eqfinp
  \end{equation*}
\end{theorem}

Second, a generic formula for the \Capra-subdifferential
of $\lzero$ is given in \cite{Chancelier-DeLara:2022_CAPRA_OPTIMIZATION}.
To state this last result, we introduce the sets
\begin{equation}
  \Dual_{\lzerovalue} = \bset{\dual \in \DUAL}{\lzerovalue \in
    \argmax_{j\in\ic{0,d}} \bp{\CoordinateNormDual{\TripleNorm{\dual}}{j} - j}} 
  \eqsepv
  \forall \lzerovalue \in \ic{0, d} \eqfinv
  \label{eq:admissible_dual}
\end{equation}
where $\na{\CoordinateNormDual{\TripleNorm{\cdot}}{j}}_{j \in \ic{1:d}}$
are the dual coordinate-$k$ norms associated with the source norm $\TripleNorm{\cdot}$, 
whose expressions are given in Definition~\ref{de:coordinate_norm}.
Also, for a nonempty closed convex set $\Convex \subseteq \DUAL$
and $\primal \in \PRIMAL$, we denote by
$\NormalCone_\Convex(\primal)$ the \emph{normal cone of~$\Convex$}
at $\primal$,
whose definition~\eqref{eq:normal_cone} and properties are recalled in Appendix~\ref{app:background_on_norms}.

\begin{proposition}[from \cite{Chancelier-DeLara:2022_CAPRA_OPTIMIZATION}, Proposition~4.7
  and \cite{Chancelier-DeLara:2021_SVVA}, Proposition~1]
  \label{pr:capra_subdifferential_any_norm}%
  Let $\TripleNorm{\cdot}$ be the source norm
  employed for the \Capra\ coupling $\CouplingCapra$
  in Definition~\ref{de:Capra}.
  Let $\na{\CoordinateNorm{\TripleNorm{\cdot}}{j}}_{j \in \ic{1:d}}$
  and $\na{\CoordinateNormDual{\TripleNorm{\cdot}}{j}}_{j \in \ic{1:d}}$
  be the associated sequences
  of coordinate-k and dual coordinate-k norms, as in Definition~\ref{de:coordinate_norm},
  and let $\na{\CoordinateNorm{\BALL}{j}}_{j \in \ic{1:d}}$
  and $\na{\CoordinateNormDual{\BALL}{j}}_{j \in \ic{1:d}}$
  be the corresponding sequences of unit balls for these norms.
  The \Capra-subdifferential of the function $\lzero$ is the closed convex set given by
  \begin{subequations}\label{eq:capra_subdifferential_all_R}
    \begin{description}
    \item $\bullet$ if
      $\primal = 0$, 
      \begin{equation}
        \partial_{\CouplingCapra} \lzero(0) 
        = 
        \bigcap_{j \in \ic{1,d}} j \BALL_{(j),\star}^{\mathrm{\FlatRR}} \eqfinv
        \label{eq:capra_subdifferential_at_0}%
      \end{equation}
    \item $\bullet$ 
      if $\primal \neq 0$ and $\lzero(\primal) = \lzerovalue$,
      \begin{equation}
        \partial_{\CouplingCapra} \lzero(\primal) 
        =
        \NormalCone_{\BALL_{(l)}^{\mathrm{\FlatRR}}}
        \np{\frac{\primal}{\TripleNorm{\primal}_{(l)}^{\mathrm{\FlatRR}}}}
        \cap
        \Dual_{\lzerovalue}
        \eqfinp
        \label{eq:capra_subdifferential}%
      \end{equation}
    \end{description}
  \end{subequations}
\end{proposition}

\section{\Capra-subdifferential of $\lzero$ for the $\ell_p$ source norms}
\label{sec:capra_subdifferential_l0}

The main contribution of this article is the following
Theorem~\ref{th:capra_convexity_and_subdifferentiability}.
It provides
explicit formulas for the \Capra-subdifferential of the $\lzero$ pseudonorm,
as introduced in~\eqref{eq:Capra_subdifferential}
and as characterized in Proposition~\ref{pr:capra_subdifferential_any_norm}
for the $\ell_p$ source norms $\TripleNorm{\cdot} = \Norm{\cdot}_p$
when $p \in \nc{1, \infty}$.

We need to introduce the following norms and notations. 
For $\dual \in \RR^d$,
if $\nu$ is a permutation of $\ic{1, d}$ such that $\abs{\dual_{\nu(1)}} \geq \ldots \geq \abs{\dual_{\nu(d)}}$,
the top $\np{k,q}$-norm $\TopNorm{\Norm{\cdot}}{k, q}$, for $k\in \ic{1,d}$,
is given by 
\begin{equation}
  \TopNorm{\Norm{\dual}}{k, q} 
  =
  \Bp{\sum_{i=1}^k \abs{\dual_{\nu(i)}}^q}^\frac{1}{q}
  \eqsepv
  \text{if} \quad q \in \rightopen{1, \infty}
  \eqsepv
  \text{and} \quad
  \TopNorm{\Norm{\dual}}{k, \infty}
  = \Norm{\dual}_\infty
  \eqfinv
  \label{eq:top_norm_explicit}%
\end{equation}
and the $\np{p,k}$-support norm $\OriginalSupportNorm{\Norm{\cdot}}{p, k}$
is the dual norm of the top $\np{k,q}$-norm $\TopNorm{\Norm{\cdot}}{k, q}$,
as defined in~\cite[\S8.1]{McDonald-Pontil-Stamos:2016}.
Besides,
for any \( \primal \in \RR^d \) 
and subset \( K \subseteq \na{1,\ldots,d} \), 
we denote by
\( \primal_K \in \RR^d \) the vector which coincides with~\( \primal \),
except for the components outside of~$K$ that vanish:
\( \primal_K \) is the orthogonal projection of~\( \primal \) onto
the subspace\footnote{%
  Here, following notation from Game Theory, 
  we have denoted by $-K$ the complementary subset 
  of~$K$ in \( \na{1,\ldots,d} \): \( K \cup (-K) = \na{1,\ldots,d} \)
  and \( K \cap (-K) = \emptyset \).}
\begin{equation}
  \FlatRR_{K} = \RR^K \times \{0\}^{-K} =
  \bset{ \primal \in \RR^d }{ \primal_j=0 \eqsepv \forall j \not\in K } 
  \subseteq \RR^d 
  \eqfinv
  \label{eq:FlatRR}
\end{equation}
where \( \FlatRR_{\emptyset}=\{0\} \).


\begin{theorem}
  \label{th:capra_convexity_and_subdifferentiability}
  Let the source norm $\TripleNorm{\cdot} = \Norm{\cdot}_p$, 
  where $p \in \nc{1, \infty}$.  
  
\item $\bullet$
  If $p = 1$,
  the $\lzero$ pseudonorm is not
  \Capra-convex, as
  its \Capra-biconjugate is
  \begin{equation}
    \SFMbi{ \lzero }{\CouplingCapra} : \ \primal \mapsto
    \begin{cases}
      0 \eqsepv &\text{ if } \primal = 0 \eqfinv \\
      1 \eqsepv &\text{ if } \primal \neq 0 \eqfinp
    \end{cases}
    \label{eq:biconjugate_1}%
  \end{equation}
  Moreover, $\lzero$ is
  only \Capra-subdifferentiable over 
  $\dom\bp{\partial_{\CouplingCapra} \lzero} 
  = \defset{\primal \in \PRIMAL}{\lzero(\primal) \leq 1}$.
  Over this domain, the \Capra-subdifferential of~$\lzero$ 
  is given by
  \begin{equation}
    \partial_{\CouplingCapra} \lzero(0) 
    = 
    \BALL_{\Norm{\cdot}_\infty}
    \
    \text{ and }
    \ 
    \partial_{\CouplingCapra} \lzero(\primal) 
    = 
    \NormalCone_{\BALL_{\Norm{\cdot}_1}}
    \np{\frac{\primal}{\Norm{\primal}_1}}
    \cap
    \defset{\dual \in \DUAL}{\Norm{\dual}_\infty \geq 1}
    \eqsepv
    \text{ if } \ \lzero(\primal) = 1 
    \eqfinp
    \label{eq:explicit_capra_subdifferential_l1}%
  \end{equation}
  
\item $\bullet$
  If $p \in \open{1, \infty}$,
  the $\lzero$ pseudonorm is 
  \Capra-convex and \Capra-subdifferentiable everywhere,
  meaning that $\dom\bp{\partial_{\CouplingCapra} \lzero} = \RR^d$.
  Its \Capra-subdifferential is given 
  by
  \begin{subequations}
    \begin{equation}
      \partial_{\CouplingCapra} \lzero(0) 
      = 
      \BALL_{\Norm{\cdot}_\infty} \eqfinv
      \label{eq:explicit_capra_subdifferential_at_0_lp}
    \end{equation}
    and at $\primal \neq 0$, denoting $\lzerovalue = \lzero(\primal)$,
    $\support = \SupportMapping(\primal)$, and $q \in \nc{1,\infty}$
    such that $\frac{1}{p} + \frac{1}{q} = 1$,
    by
    \begin{equation}
      \dual \in
      \partial_{\CouplingCapra} \lzero(\primal) 
      \iff
      \begin{cases}
        \dual_\support \in
        \NormalCone_{\BALL_{\Norm{\cdot}_p}}
        \np{\frac{\primal}{\Norm{\primal}_p}} 
        \eqfinv \\
        \abs{\dual_j} \leq \min_{i \in \support} \abs{\dual_i} \eqsepv 
        \forall j \notin \support \eqfinv \\
        \abs{\dual_{\nu(k+1)}}^q \geq \bp{\TopNorm{\Norm{\dual}}{k, q} + 1}^q 
        - \bp{\TopNorm{\Norm{\dual}}{k, q}}^q \eqsepv 
        \forall k \in \ic{0, \lzerovalue-1} 
        \eqfinv
        \\
        \abs{\dual_{\nu(\lzerovalue+1)}}^q
        \leq 
        \bp{\TopNorm{\Norm{\dual}}{\lzerovalue, q} + 1}^q 
        - \bp{\TopNorm{\Norm{\dual}}{\lzerovalue, q}}^q
        \;(\text{when}\;  \lzerovalue\not= d) \eqfinv
      \end{cases} \label{eq:explicit_capra_subdifferential_lp}%
    \end{equation}
    where, for any $\dual \in \RR^d$,
    $\nu$ denotes a permutation of $\ic{1, d}$ such that 
    $\abs{\dual_{\nu(1)}} \geq \ldots \geq \abs{\dual_{\nu(d)}}$.
  \end{subequations}
  
\item $\bullet$
  If $p = \infty$,
  the $\lzero$ pseudonorm is not
  \Capra-convex, 
  as its \Capra-biconjugate is
  \begin{equation}
    \SFMbi{ \lzero }{\CouplingCapra} : \ \primal \mapsto
    \begin{cases}
      0 \eqsepv &\text{ if } \primal = 0 \eqfinv \\
      \frac{\Norm{\primal}_1}{\Norm{\primal}_\infty} \eqsepv &\text{ if } \primal \neq 0 \eqfinp
    \end{cases}
    \label{eq:biconjugate_infty}%
  \end{equation}
  Moreover, the function~$\lzero$ is
  only \Capra-subdifferentiable 
  over the domain
    \begin{equation}
    \dom\bp{\partial_{\CouplingCapra} \lzero} 
    = \bset{\primal\in\RR^d}%
    {\exists \lambda > 0 \eqsepv \primal_k \in \na{-\lambda, 0, \lambda } \eqsepv \forall k \in \ic{1, d} } 
    = 
  \bigcup_{\lambda > 0} \na{-\lambda, 0, \lambda}^d
  \eqfinp
    \end{equation}
  Over this domain,
  the \Capra-subdifferential of~$\lzero$ is given   by
  \begin{subequations}
    \begin{equation}
      \partial_{\CouplingCapra} \lzero(0) 
      = 
      \BALL_{\Norm{\cdot}_\infty} \eqfinv
      \label{eq:explicit_capra_subdifferential_at_0_l_infty}
    \end{equation}
    and, at $\primal \in \cup_{\lambda > 0} \na{-\lambda, 0, \lambda}^d \setminus \na{0}$,
    denoting $\lzerovalue = \lzero(\primal)$,
    $\support = \SupportMapping(\primal)$,
    by
    \begin{equation}
      \dual \in
      \partial_{\CouplingCapra} \lzero(\primal) 
      \iff
      \begin{cases}
        \dual_\support \in
        \NormalCone_{\BALL_{\Norm{\cdot}_\infty}}
        \np{\frac{\primal}{\Norm{\primal}_\infty}} 
        \eqfinv \\
        \abs{\dual_j} \leq \min_{i \in \support} \abs{\dual_i} \eqsepv 
        \forall j \notin \support \eqfinv \\
        \abs{\dual_{\nu(k+1)}} \geq 1 \eqsepv 
        \forall k \in \ic{0, \lzerovalue-1} 
        \eqfinv
        \\
        \abs{\dual_{\nu(\lzerovalue+1)}}
        \leq 1 \;(\text{when}\;  \lzerovalue\not= d) \eqfinv
      \end{cases} \label{eq:explicit_capra_subdifferential_l_infty}%
    \end{equation}
    where, for any $\dual \in \RR^d$,
    $\nu$ denotes a permutation of $\ic{1, d}$ such that 
    $\abs{\dual_{\nu(1)}} \geq \ldots \geq \abs{\dual_{\nu(d)}}$.
  \end{subequations}
\end{theorem}

We proceed in three steps to prove Theorem~\ref{th:capra_convexity_and_subdifferentiability}. 
First, in~\S\ref{sec:explicit_admissible_dual},
we provide an explicit description of the set~$\Dual_{\lzerovalue}$
in~\eqref{eq:admissible_dual} (that appears in~\eqref{eq:capra_subdifferential}).
Second, in~\S\ref{sec:explicit_normal_cone}, 
we provide an explicit expression 
for the normal cone~$\NormalCone_{\BALL_{(l)}^{\mathrm{\FlatRR}}}$
in~\eqref{eq:capra_subdifferential}.
Third, in~\S\ref{subsec:Proof_of_Theorem}, 
we apply both results to the
generic formulation
of the \Capra-subdifferential of $\lzero$ given in~\eqref{eq:capra_subdifferential_all_R},
and wrap up the proof
of Theorem~\ref{th:capra_convexity_and_subdifferentiability}.


We will need the following 
properties
of the coordinate-$k$ and dual coordinate-$k$ norms
of Definition~\ref{de:coordinate_norm}.

\begin{proposition}[from \cite{Chancelier-DeLara:2022_CAPRA_OPTIMIZATION}, Table~1]
  \label{pr:table_norms}
  Let the source norm $\TripleNorm{\cdot}$ be a $\ell_p$ norm
  with $p \in \nc{1, \infty}$,
  and let $q \in \nc{1, \infty}$ 
  such that $\frac{1}{p} + \frac{1}{q} = 1$.
  The coordinate-$k$ and dual coordinate-$k$ norms
  in Definition~\ref{de:coordinate_norm}
  are given, for $k \in \ic{1, d}$, by
  \begin{equation}
    \CoordinateNormDual{\TripleNorm{\cdot}}{k} 
    = \TopNorm{\Norm{\cdot}}{k, q}
    \quad \text{and} \quad
    \CoordinateNorm{\TripleNorm{\cdot}}{k} 
    = \OriginalSupportNorm{\Norm{\cdot}}{p, k}
    \eqfinp
    \label{eq:table_norms}
  \end{equation}
\end{proposition}

\subsection{Description of the sets~$\Dual_\lzerovalue$}
\label{sec:explicit_admissible_dual}%

We derive explicit descriptions of the sets $\Dual_{\lzerovalue}$ in~\eqref{eq:admissible_dual} 
for the $\ell_p$ source norms 
$\TripleNorm{\cdot} = \Norm{\cdot}_p$, when $p \in \nc{1, \infty}$.
We start with two preliminary results on the top $\np{k,q}$-norm $\TopNorm{\Norm{\cdot}}{k, q}$,
whose expression is given in~\eqref{eq:top_norm_explicit}.
We state our first preliminary result in
Lemma~\ref{le:admissible_dual_inequality}.

\begin{lemma}
  \label{le:admissible_dual_inequality}%
  Let $\dual \in \DUAL$, $q \in \rightopen{1, \infty}$ and $k \in \ic{0, d-1}$.
  We have that
  \begin{equation}
    \TopNorm{\Norm{\dual}}{k+1, q} - 1 \leq \TopNorm{\Norm{\dual}}{k, q}  
    \implies \TopNorm{\Norm{\dual}}{k+j, q} - j \leq  \TopNorm{\Norm{\dual}}{k, q} \eqsepv  
    \forall j \in \ic{1, d-k} \eqfinp
    \label{eq:admissible_dual_inequality}
  \end{equation}
  Moreover, the same result holds if inequalities are strict in~\eqref{eq:admissible_dual_inequality}.
\end{lemma}

\begin{proof}
  Let $\dual \in \DUAL$ and $\nu$ a permutation of $\ic{1, d}$
  such that $\abs{\dual_{\nu(1)}} \geq \ldots \geq \abs{\dual_{\nu(d)}}$.
  Let also $q \in \rightopen{1, \infty}$, $k \in \ic{0, d-1}$ and $j \in \ic{1, d-k}$. 
  We use the shorthand notation 
  \begin{equation}
    \dual_{k,q}^\Sigma = \sum_{i=1}^k \abs{\dual_{\nu(i)}}^q \eqfinv
    \label{eq:Sigma}%
  \end{equation}
  so that, from Proposition~\ref{pr:table_norms}, we have that
  $\TopNorm{\Norm{\dual}}{k, q} = \bp{\dual_{k,q}^\Sigma}^{\frac{1}{q}}$.
  
  First,
  we prove the inequality
  \begin{equation}
    \bp{\dual_{k,q}^\Sigma + j \abs{\dual_{\nu(k+1)}}^q}^{\frac{1}{q}}
    - \bp{\dual_{k,q}^\Sigma}^{\frac{1}{q}}
    \leq
    j \Bc{\bp{\dual_{k,q}^\Sigma + \abs{\dual_{\nu(k+1)}}^q}^{\frac{1}{q}}
      - \bp{\dual_{k,q}^\Sigma}^{\frac{1}{q}}} \eqfinp
    \label{eq:inequality}%
  \end{equation}
  Indeed, we have that
  \begin{align*}
    \frac{1}{j}\bp{\dual_{k,q}^\Sigma + j \abs{\dual_{\nu(k+1)}}^q}^{\frac{1}{q}}
    + \bp{1 - \frac{1}{j}} \bp{\dual_{k,q}^\Sigma}^{\frac{1}{q}}
    &\leq
      \Bp{\frac{1}{j}\bp{\dual_{k,q}^\Sigma + j \abs{\dual_{\nu(k+1)}}^q}
      + \bp{1 - \frac{1}{j}} \dual_{k,q}^\Sigma}^{\frac{1}{q}} \eqfinv
      \intertext{by concavity of the 
      function~$x \mapsto x^{\frac{1}{q}}$ on $\RR_+$ for $q \geq 1$, }
      \implies 
      \bp{\dual_{k,q}^\Sigma + j \abs{\dual_{\nu(k+1)}}^q}^{\frac{1}{q}}
      + \bp{j - 1} \bp{\dual_{k,q}^\Sigma}^{\frac{1}{q}}
    &\leq
      j \Bp{\dual_{k,q}^\Sigma + \abs{\dual_{\nu(k+1)}}^q}^{\frac{1}{q}} \eqfinv
    \\
    \implies 
    \bp{\dual_{k,q}^\Sigma + j \abs{\dual_{\nu(k+1)}}^q}^{\frac{1}{q}}
    - \bp{\dual_{k,q}^\Sigma}^{\frac{1}{q}}
    &\leq
      j \Bc{\bp{\dual_{k,q}^\Sigma + \abs{\dual_{\nu(k+1)}}^q}^{\frac{1}{q}}
      - \bp{\dual_{k,q}^\Sigma}^{\frac{1}{q}}} \eqfinp
  \end{align*}
  
  Second, we prove the implication in~\eqref{eq:admissible_dual_inequality}
  in its nonstrict inequality version. 
  Let us assume that
  $\TopNorm{\Norm{\dual}}{k+1, q} - 1 \leq \TopNorm{\Norm{\dual}}{k, q}$.
  By definition of $\dual_{k,q}^\Sigma$ in~\eqref{eq:Sigma} and
  since $\abs{\dual_{\nu(k+1)}} \geq \abs{\dual_{\nu(k+2)}}
  \geq \ldots \geq \abs{\dual_{\nu(k+j)}} $,
  we have that 
  \begin{align*}
    \TopNorm{\Norm{\dual}}{k+j, q}
     - \TopNorm{\Norm{\dual}}{k, q}
     & \leq
      \bp{\dual_{k,q}^\Sigma + j \abs{\dual_{\nu(k+1)}}^q}^{\frac{1}{q}}
      - \bp{\dual_{k,q}^\Sigma}^{\frac{1}{q}} \eqfinv
      \\
    &\leq 
      j \Bc{\bp{\dual_{k,q}^\Sigma + \abs{\dual_{\nu(k+1)}}^q}^{\frac{1}{q}}
      - \bp{\dual_{k,q}^\Sigma}^{\frac{1}{q}}} \eqfinv
      \tag{from \eqref{eq:inequality}}
\\ 
    &=
      j \Bc{\TopNorm{\Norm{\dual}}{k+1, q} - \TopNorm{\Norm{\dual}}{k, q}} \eqfinv
      \tag{from the expression of $\TopNorm{\Norm{\cdot}}{k, q}$
      in~\eqref{eq:top_norm_explicit} and by~\eqref{eq:Sigma} }
    \\
    &\leq j
      \tag{by the assumption that $\TopNorm{\Norm{\dual}}{k+1, q} - \TopNorm{\Norm{\dual}}{k, q} \leq 1$}
      \eqfinv
  \end{align*}
  which proves that $\TopNorm{\Norm{\dual}}{k+j, q} - j \leq  \TopNorm{\Norm{\dual}}{k, q}$.
  The proof of the strict inequality version of~\eqref{eq:admissible_dual_inequality}
  is analogous.
  \medskip
  
  This ends the proof.
\end{proof}

We state our second preliminary result in
Lemma~\ref{le:top_norm_inequality}.

\begin{lemma}
  \label{le:top_norm_inequality}%
  Let $\dual \in \DUAL$, $q \in \rightopen{1, \infty}$ and $k \in \ic{0, d-1}$.
  We have that
  \begin{equation}
    \TopNorm{\Norm{\dual}}{k+1, q} - 1 \leq \TopNorm{\Norm{\dual}}{k, q} 
    \iff \abs{\dual_{\nu(k+1)}}^q \leq 
    \bp{\TopNorm{\Norm{\dual}}{k, q} + 1}^q 
    - \bp{\TopNorm{\Norm{\dual}}{k, q}}^q 
    \eqfinp
    \label{eq:top_norm_inequality}%
  \end{equation}
  Moreover, the same result holds if inequalities are strict or replaced
  with equalities in~\eqref{eq:top_norm_inequality}.
\end{lemma}

\begin{proof}
  For $\dual \in \DUAL$ and $k \in \ic{0, d-1}$, we have that
  \begin{align*}
    \TopNorm{\Norm{\dual}}{k+1, q} - 1 \leq \TopNorm{\Norm{\dual}}{k, q} 
    &\iff 
      \Bp{\sum_{i=1}^k \abs{\dual_{\nu(i)}}^q + \abs{\dual_{\nu(k+1)}}^q}^\frac{1}{q}
      -1 \leq \TopNorm{\Norm{\dual}}{k, q} \eqfinv
      \tag{from the expression of $\TopNorm{\Norm{\cdot}}{k, q}$
      in~\eqref{eq:top_norm_explicit}}
    \\
    &\iff 
      \sum_{i=1}^k \abs{\dual_{\nu(i)}}^q + \abs{\dual_{\nu(k+1)}}^q
      \leq \bp{\TopNorm{\Norm{\dual}}{k, q} + 1}^q \eqfinv
      \tag{as the function $x \mapsto x^q$ is nondecreasing on $\RR_+$}
  \end{align*}
  so that finally, by definition~\eqref{eq:top_norm_explicit} of
  $\TopNorm{\Norm{\cdot}}{k, q}$, we get 
  \begin{equation*}
    \TopNorm{\Norm{\dual}}{k+1, q} - 1 \leq \TopNorm{\Norm{\dual}}{k, q} 
    \iff \abs{\dual_{\nu(k+1)}}^q \leq 
    \bp{\TopNorm{\Norm{\dual}}{k, q} + 1}^q 
    - \bp{\TopNorm{\Norm{\dual}}{k, q}}^q 
    \eqfinp
  \end{equation*}
  The proof of the strict inequality and equality versions
  of~\eqref{eq:admissible_dual_inequality} is analogous.
\end{proof}

We now provide explicit expressions of the sets~$\Dual_{\lzerovalue}$ 
in~\eqref{eq:admissible_dual} for the $\ell_p$ source norms 
$\TripleNorm{\cdot} = \Norm{\cdot}_p$, when $p \in \nc{1, \infty}$.

\begin{proposition}
  \label{pr:admissible_dual_equivalence}%
  Let the source norm be the $\ell_p$ norm $\TripleNorm{\cdot} = \Norm{\cdot}_p$,
  where $p \in \nc{1, \infty}$, and let $q \in \nc{1, \infty}$ 
  be such that $\frac{1}{p} + \frac{1}{q} = 1$. 
  For $\lzerovalue \in \ic{0, d}$, the set $\Dual_{\lzerovalue}$ 
   in~\eqref{eq:admissible_dual} is given by
  \begin{subequations}
    \begin{description}
    \item
      $\bullet$ if $p=1$,
      \begin{equation}
        \Dual_l = \begin{cases}
          \BALL_{\Norm{\cdot}_\infty} 
          & \text{ if } \lzerovalue = 0 \eqfinv \\
          \defset{\dual \in \RR^d}{\Norm{\dual}_\infty \geq 1}
          & \text{ if } \lzerovalue = 1 \eqfinv \\
          \emptyset & \text{ else,}
        \end{cases} \label{eq:admissible_dual_l1}%
      \end{equation} \\
    \item
      $\bullet$ if $p \in \leftopen{1, \infty}$,
      for $\dual \in \DUAL$ and $\nu$ a permutation of $\ic{1, d}$ 
      such that $\abs{\dual_{\nu(1)}} \geq \ldots \geq \abs{\dual_{\nu(d)}}$, 
      \begin{equation}
        \dual \in \Dual_{\lzerovalue} \iff 
        \begin{cases}
          \abs{\dual_{\nu(k+1)}}^q \geq \bp{\TopNorm{\Norm{\dual}}{k, q} + 1}^q 
          - \bp{\TopNorm{\Norm{\dual}}{k, q}}^q \eqsepv 
          \forall k \in \ic{0, \lzerovalue-1} 
          \eqfinv
          \\
          \abs{\dual_{\nu(\lzerovalue+1)}}^q
          \leq 
          \bp{\TopNorm{\Norm{\dual}}{\lzerovalue, q} + 1}^q 
          - \bp{\TopNorm{\Norm{\dual}}{\lzerovalue, q}}^q
          \;(\text{when}\;  \lzerovalue\not= d)
          \eqfinp
        \end{cases} 
        \label{eq:admissible_dual_equivalence}%
      \end{equation}
    \end{description}
  \end{subequations}
\end{proposition}

\begin{proof}
  We consider the case $p=1$.
  When the source norm is $\TripleNorm{\cdot} = \Norm{\cdot}_1$,
  we have that for $k \in \ic{1, d}$, 
  $\CoordinateNormDual{\TripleNorm{\cdot}}{k} = \Norm{\cdot}_{\infty}$
  \citep[Table 1]{Chancelier-DeLara:2022_CAPRA_OPTIMIZATION},
  and that $\CoordinateNormDual{\TripleNorm{\cdot}}{0} = 0$,
  following the convention introduced in~\cite[\S3.2]{Chancelier-DeLara:2022_CAPRA_OPTIMIZATION}.
  Therefore, from the expression of $\Dual_\lzerovalue$
  in~\eqref{eq:admissible_dual}, we get that
  \[ 
    \Dual_{\lzerovalue} = \defset{\dual \in \DUAL}{\lzerovalue \in \argmax_{j\in\ic{0,d}} 
      \bp{ \Norm{\dual}_{\infty} \boldsymbol{1}_{j \neq 0} - j}}
    = \begin{cases}
      \BALL_{\Norm{\cdot}_\infty} 
      & \text{ if } \lzerovalue = 0 \eqfinv \\
      \defset{\dual \in \RR^d}{\Norm{\dual}_\infty \geq 1}
      & \text{ if } \lzerovalue = 1 \eqfinv \\
      \emptyset & \text{ else,}
    \end{cases} 
  \]
  hence \eqref{eq:admissible_dual_l1}. 
  
  Next, we consider $p \in \leftopen{1, \infty}$,
  and proceed in two steps to prove the equivalence in~\eqref{eq:admissible_dual_equivalence}.
  
  In the first step $(\impliedby)$, 
  we take $\dual \in \DUAL$ and we consider the two following
  cases that correspond to the right-hand side in~\eqref{eq:admissible_dual_equivalence}.
  \medbreak
  $\bullet$ If $\abs{\dual_{\nu(k+1)}}^q \geq \bp{\TopNorm{\Norm{\dual}}{k, q} + 1}^q 
  - \bp{\TopNorm{\Norm{\dual}}{k, q}}^q \eqsepv \forall k \in \ic{0, \lzerovalue-1}
  \eqfinv$ 
  \medbreak
  then         
  we get that
  \begin{align*}
    \TopNorm{\Norm{\dual}}{k+1, q}
    & - 1 
      \geq 
      \TopNorm{\Norm{\dual}}{k, q} \eqsepv 
      \forall k \in \ic{0, \lzerovalue-1}  \eqsepv
      \tag{from~\eqref{eq:top_norm_inequality}}
    \\
    &\implies \TopNorm{\Norm{\dual}}{k+1, q} - (k+1) 
      \geq 
      \TopNorm{\Norm{\dual}}{k, q} - k \eqsepv 
      \forall k \in \ic{0, \lzerovalue-1}  \eqsepv \\
    &\implies \lzerovalue 
      \in 
      \argmax_{j\in\ic{0,l}} \bp{\TopNorm{\Norm{\dual}}{j, q} - j} \eqfinp 
  \end{align*}
  \medbreak

  $\bullet$ If $\lzerovalue\not=d$ and $\abs{\dual_{\nu(\lzerovalue+1)}}^q 
  \leq 
  \bp{\TopNorm{\Norm{\dual}}{\lzerovalue, q} + 1}^q 
  - \bp{\TopNorm{\Norm{\dual}}{\lzerovalue, q}}^q \eqfinv$ 
  \medbreak
  then 
  we get that
  \begin{align*}
    \TopNorm{\Norm{\dual}}{\lzerovalue+1, q}
    & - 1 
      \leq 
      \TopNorm{\Norm{\dual}}{\lzerovalue, q} \eqfinv
      \tag{from~\eqref{eq:top_norm_inequality}}
    \\
    &\implies 
      \TopNorm{\Norm{\dual}}{\lzerovalue+j, q} - j 
      \leq 
      \TopNorm{\Norm{\dual}}{\lzerovalue, q} 
      \eqsepv  
      \forall j \in \ic{1, d-\lzerovalue} 
      \eqfinv
      \tag{from~\eqref{eq:admissible_dual_inequality}}
    \\
    &\implies 
      \TopNorm{\Norm{\dual}}{\lzerovalue+j, q} - (\lzerovalue + j) 
      \leq 
      \TopNorm{\Norm{\dual}}{\lzerovalue, q} - \lzerovalue 
      \eqsepv  
      \forall j \in \ic{1, d-\lzerovalue} 
      \eqfinv 
    \\
    &\implies 
      \lzerovalue 
      \in 
      \argmax_{j\in\ic{l,d}} 
      \bp{\TopNorm{\Norm{\dual}}{j, q} - j} 
      \eqfinp
  \end{align*}
  Therefore, if the vector~$\dual$ satisfies both of the above assumptions,
  we get that 
  $\lzerovalue \in 
  \argmax_{j\in\ic{0,d}} 
  \bp{\TopNorm{\Norm{\dual}}{j, q} - j} $, and hence that
  $\dual \in \Dual_\lzerovalue$ by~\eqref{eq:admissible_dual}. This concludes the first step.
  \medskip
  
  In the second step $(\implies)$,
  we proceed by contraposition, assuming that either one of the two
  inequalities in the right-hand side of~\eqref{eq:admissible_dual_equivalence}
  is not satisfied. 
  \medbreak
  $\bullet$ If $\exists k \in \ic{0, \lzerovalue-1} \eqsepv \abs{\dual_{\nu(k+1)}}^q 
  < \bp{\TopNorm{\Norm{\dual}}{k, q} + 1}^q 
  - \bp{\TopNorm{\Norm{\dual}}{k, q}}^q \eqfinv$ 
  \medbreak
  then 
  we get that
  \begin{align*}
    \exists k
    &\in
      \ic{0, \lzerovalue-1} \eqsepv
      \TopNorm{\Norm{\dual}}{k+1, q} - 1 < \TopNorm{\Norm{\dual}}{k, q} \eqfinv
      \tag{from~\eqref{eq:top_norm_inequality} {with strict inequality}}
    \\
    &\implies
    \exists k \in \ic{0, \lzerovalue-1} \eqsepv
    \TopNorm{\Norm{\dual}}{k+j, q} - j < \TopNorm{\Norm{\dual}}{k, q}  
    \eqsepv
    \forall j \in \ic{1, d-k} \eqfinv 
    \tag{from~\eqref{eq:admissible_dual_inequality}}
    \\
    &\implies
      \exists k \in \ic{0, \lzerovalue-1} \eqsepv
      \TopNorm{\Norm{\dual}}{k+j, q} - (k+j) < \TopNorm{\Norm{\dual}}{k, q} - k 
      \eqsepv
      \forall j \in \ic{1, d-k} \eqfinv 
    \\
    &\implies
      \exists k \in \ic{0, \lzerovalue-1} \eqsepv
      \TopNorm{\Norm{\dual}}{\lzerovalue, q} - \lzerovalue < \TopNorm{\Norm{\dual}}{k, q} - k 
      \eqsepv
      \tag{as $\lzerovalue \in \nset{k+j}{j \in \ic{1, d-k}}$ since $k< l$}
    \\
    &\implies
      \exists k \in \ic{0, d} \eqsepv
      \TopNorm{\Norm{\dual}}{\lzerovalue, q} - \lzerovalue < \TopNorm{\Norm{\dual}}{k, q} - k 
      \eqsepv
      \tag{as $\ic{0, \lzerovalue-1}\subset \ic{0,d}$}
    \\
    &\implies
      \lzerovalue \notin \argmax_{j\in\ic{0,d}} 
      \bp{\TopNorm{\Norm{\dual}}{j, q} - j} \eqfinp
  \end{align*}
  \medbreak
  $\bullet$ If $\lzerovalue\not= d$ and $\abs{\dual_{\nu(\lzerovalue+1)}}^q > 
  \bp{\TopNorm{\Norm{\dual}}{\lzerovalue, q} + 1}^q 
  - \bp{\TopNorm{\Norm{\dual}}{\lzerovalue, q}}^q \eqfinv$
  \medbreak
  then we get that 
  \begin{align*}
    \TopNorm{\Norm{\dual}}{\lzerovalue+1, q}
    & - 1 
      > 
      \TopNorm{\Norm{\dual}}{\lzerovalue, q} 
      \eqsepv
      \tag{from~\eqref{eq:top_norm_inequality} {with strict inequality}}
    \\
    &\implies 
      \TopNorm{\Norm{\dual}}{\lzerovalue+1, q} - (\lzerovalue+1) 
      > 
      \TopNorm{\Norm{\dual}}{\lzerovalue, q} - \lzerovalue 
      \eqsepv  
    \\
    &\implies 
      \lzerovalue \notin \argmax_{j\in\ic{0,d}} 
      \bp{\TopNorm{\Norm{\dual}}{j, q} - j} 
      \eqfinp
  \end{align*}
  In either case, we get that $\dual \notin \Dual_{\lzerovalue}$
  by~\eqref{eq:admissible_dual}, which
  concludes the second step.
  We have finally proved
  the equivalence in~\eqref{eq:admissible_dual_equivalence}.
  \medskip
  
  This ends the proof.
\end{proof}

\subsection{Expression 
  of the normal cone~$\NormalCone_{\OriginalSupportNorm{\BALL}{p, \lzerovalue}}$}
\label{sec:explicit_normal_cone}

We now turn to giving a description of the normal cone~$\NormalCone_{\OriginalSupportNorm{\BALL}{p, \lzerovalue}}$
in~\eqref{eq:capra_subdifferential}
for the $\ell_p$ source norms 
$\TripleNorm{\cdot} = \Norm{\cdot}_p$, when $p \in \nc{1, \infty}$.
We start with the following Lemma~\ref{le:normal_cone}.

\begin{lemma}
  \label{le:normal_cone}%
  Let the source norm be the $\ell_p$ norm $\TripleNorm{\cdot} = \Norm{\cdot}_p$,
  where $p \in \nc{1, \infty}$.
  Let $\primal \in \RR^d$, $\lzerovalue = \lzero(\primal)$, $\support = \SupportMapping(\primal)$. 
  If $\lzerovalue \in \ic{1, d}$, we have that

  \begin{subequations}
    \begin{align}
      \Norm{\frac{\primal}{\OriginalSupportNorm{\Norm{\primal}}{p, \lzerovalue}}}_p 
      &= 1 \eqfinv
        \label{eq:F2}
      \\
      \dual \in \NormalCone_{\OriginalSupportNorm{\BALL}{p, \lzerovalue}}
      \Bp{\frac{\primal}{\OriginalSupportNorm{\Norm{\primal}}{p, \lzerovalue}}}
      &\iff 
        \TopNorm{\Norm{\dual}}{\lzerovalue,q}
        = \proscal{\frac{\primal}{\OriginalSupportNorm{\Norm{\primal}}{p, \lzerovalue}}}{\dual_\support}
        \eqfinv
        \label{eq:F1} 
      \\
      \dual \in \NormalCone_{\OriginalSupportNorm{\BALL}{p, \lzerovalue}}
      \Bp{\frac{\primal}{\OriginalSupportNorm{\Norm{\primal}}{p, \lzerovalue}}}
      &\implies
        \TopNorm{\Norm{\dual}}{\lzerovalue,q} 
        \leq \Norm{\dual_\support}_q
        \label{eq:F3}
        \eqfinp
    \end{align}
  \end{subequations}
\end{lemma}

\begin{proof}
  Let $q \in \nc{1, \infty}$ 
  be such that $\frac{1}{p} + \frac{1}{q} = 1$.
  Let $\primal \in \RR^d$, $\lzerovalue = \lzero(\primal)$ and
  $\support = \SupportMapping(\primal)$. {As we assume that $l\in \ic{1, d}$, we have that
    $\primal\not= 0$ and we set $\primalbis = \frac{\primal}{\OriginalSupportNorm{\Norm{\primal}}{p, \lzerovalue}}$.}
    
  First, we prove~\eqref{eq:F2}. {We have that $\lzerovalue \geq 1$ and
  $\lzero(\primalbis) = \lzero(\primal) = \lzerovalue$. Thus,
  using~\cite[Proposition~3.5]{Chancelier-DeLara:2022_CAPRA_OPTIMIZATION},
  we obtain that $\TripleNorm{\primalbis} =
  \CoordinateNorm{\TripleNorm{\primalbis}}{\lzerovalue}$}.
  Thus, from Proposition~\ref{pr:table_norms} we deduce that
  $\Norm{\primalbis}_p = \OriginalSupportNorm{\Norm{\primalbis}}{p,l} = 1$.
  
  Second, we prove~\eqref{eq:F1}.
  We have the equivalence
  \begin{align}
    \dual \in \NormalCone_{\OriginalSupportNorm{\BALL}{p, \lzerovalue}}\np{\primalbis}
    &\iff 
      \OriginalSupportNorm{\Norm{\primalbis}}{p, \lzerovalue}
      \TopNorm{\Norm{\dual}}{\lzerovalue,q}
      = \proscal{\primalbis}{\dual}
      \tag{by definition~\eqref{eq:couple_TripleNorm-dual_and_normal_cone} of the normal cone}\\
    &\iff
      \TopNorm{\Norm{\dual}}{\lzerovalue,q}
      = \proscal{\primalbis}{\dual_\support}
      \eqfinp
      \tag{from $\OriginalSupportNorm{\Norm{\primalbis}}{p,\lzerovalue} = 1$
      and $\support = \SupportMapping(\primalbis)$}
  \end{align}
  
  Third, we prove~\eqref{eq:F3}.
  We have that
  \begin{align}
    \dual \in \NormalCone_{\OriginalSupportNorm{\BALL}{p, \lzerovalue}}\np{\primalbis}
    &\iff
      \TopNorm{\Norm{\dual}}{\lzerovalue,q}
      = \proscal{\primalbis}{\dual_\support}_{\RR^\lzerovalue}
      \tag{from \eqref{eq:F1}}
    \\
    &\implies
      \TopNorm{\Norm{\dual}}{\lzerovalue,q} 
      \leq \Norm{\dual_\support}_q
      \eqfinp
      \tag{from the H{\"o}lder inequality and \eqref{eq:F2}}
  \end{align}
  
  This ends the proof.
\end{proof}

We now provide an explicit expression of the normal cone
in~\eqref{eq:capra_subdifferential}
for the $\ell_p$ source norms 
$\TripleNorm{\cdot} = \Norm{\cdot}_p$, when $p \in \nc{1, \infty}$.

\begin{proposition}
  Let the source norm be the $\ell_p$ norm $\TripleNorm{\cdot} = \Norm{\cdot}_p$,
  where $p \in \nc{1, \infty}$.
  Let $\primal \in \RR^d$, $\lzerovalue = \lzero(\primal)$ and $\support = \SupportMapping(\primal)$.
  If $\lzerovalue \in \ic{1, d}$, 
  we have that
  \begin{equation}
    \dual \in \NormalCone_{\OriginalSupportNorm{\BALL}{p, \lzerovalue}} \Bp{\frac{\primal}{\OriginalSupportNorm{\Norm{\primal}}{p, \lzerovalue}}}
    \iff
    \begin{cases}
      \dual_\support \in
      \NormalCone_{\BALL_{\Norm{\cdot}_p}}
      \np{\frac{\primal}{\Norm{\primal}_p}} 
      \eqfinv \\
      \abs{\dual_j} \leq \min_{i \in \support} \abs{\dual_i} \eqsepv 
      \forall j \notin \support \eqfinp
    \end{cases}\label{eq:normal_cone_equivalence}
  \end{equation}
  \label{pr:normal_cone}
\end{proposition}

\begin{proof}
  Let $q \in \nc{1, \infty}$ 
  be such that $\frac{1}{p} + \frac{1}{q} = 1$.
  Let $\primal \in \RR^d$, $\lzerovalue = \lzero(\primal)$ and
  $\support = \SupportMapping(\primal)$,
  and let us set 
  $\primalbis = \frac{\primal}{\OriginalSupportNorm{\Norm{\primal}}{p,
      \lzerovalue}}$.
  Let $\dual \in \RR^d$, and let us set $\dualsupport = \SupportMapping(\dual)$.
  
  First, we prove that
  \begin{equation}
    \dual \in \NormalCone_{\OriginalSupportNorm{\BALL}{p,\lzerovalue}}
    \np{\primalbis} 
    \implies 
    \TopNorm{\Norm{\dual}}{l, q} = \Norm{\dual_\support}_q
    \eqfinp
    \label{eq:CN_normal_cone}
  \end{equation}
  We consider two cases.
  In the first case, we assume that $\cardinal{\dualsupport} =
  \cardinal{\SupportMapping(\dual)}
  \leq \cardinal{\SupportMapping(\primal)}=\cardinal{\support} = l$.
  Since the vector~$\dual$ has at most $l$~nonzero coordinates, 
  we get that $\TopNorm{\Norm{\dual}}{l, q} = \Norm{\dual_\dualsupport}_q$
    from the expression~\eqref{eq:top_norm_explicit} of $\TopNorm{\Norm{\cdot}}{l, q}$.
  It follows that
  \begin{align*}
    \dual \in \NormalCone_{\OriginalSupportNorm{\BALL}{p,\lzerovalue}}
    \np{\primalbis}
    &\implies \Norm{\dual_\dualsupport}_q \leq
      \Norm{\dual_\support}_q
      \eqfinv
      \tag{from \eqref{eq:F3}}
    \\
    &\implies
      \Norm{\dual_\support}_q
      =
      \Norm{\dual_\dualsupport}_q
      =
      \TopNorm{\Norm{\dual}}{l, q}
      \eqfinp
      \tag{from $\Norm{\dual_\support}_q
      \leq\Norm{\dual}_q=\Norm{\dual_\dualsupport}_q$, because
      \( \cardinal{\dualsupport} =        \cardinal{\SupportMapping(\dual)} \)}
  \end{align*}
  In the second case, we assume that 
  $\cardinal{\dualsupport} =
  \cardinal{\SupportMapping(\dual)} > \cardinal{\SupportMapping(\primal)}=\cardinal{\support} = l$.
  Since the vector~$\dual$ has more than~$l$ nonzero coordinates, 
  we get that	
  $\TopNorm{\Norm{\dual}}{l, q} \geq \Norm{\dual_\support}_q$
  from the expression~\eqref{eq:top_norm_explicit} of~$\TopNorm{\Norm{\cdot}}{l, q}$.
  Combined with~\eqref{eq:F3}, we deduce that
  $\dual \in \NormalCone_{\OriginalSupportNorm{\BALL}{p,\lzerovalue}}
  \np{\primalbis}
  \implies
  \TopNorm{\Norm{\dual}}{\lzerovalue,q} 
  = \Norm{\dual_\support}_q$.
  Gathering the conclusions of both cases, we obtain~\eqref{eq:CN_normal_cone}.
  
  Second, we prove~\eqref{eq:normal_cone_equivalence}.
  Observing that $\Norm{\primalbis}_p = 1$ from~\eqref{eq:F2}, we have that
  \begin{align*}
    \dual \in \NormalCone_{\OriginalSupportNorm{\BALL}{p,\lzerovalue}}
    \np{\primalbis}
    &\iff 
      \begin{cases}
	\Norm{\primalbis}_p
	\Norm{\dual_\support}_q
	= \proscal{\primalbis}{\dual_\support}
	\eqfinv \\
	\TopNorm{\Norm{\dual}}{l, q} = \Norm{\dual_\support}_q
	\eqfinv
      \end{cases}
    \tag{$\implies$ from \eqref{eq:F1}, \eqref{eq:CN_normal_cone}; 
    $\impliedby$ from \eqref{eq:F1}}
    \\
    &\iff
      \begin{cases}
	\dual_\support \in
	\NormalCone_{\BALL_{\Norm{\cdot}_p}}
	\np{\frac{\primal}{\Norm{\primal}_p}} 
	\eqfinv \\
	\abs{\dual_j} \leq \min_{i \in \support} \abs{\dual_i} \eqsepv 
	\forall j \notin \support \eqfinv
      \end{cases}
  \end{align*}
  by definition~\eqref{eq:couple_TripleNorm-dual_and_normal_cone} 
  of the normal cone, observing that
  $\primalbis = \frac{\primal}{\Norm{\primal}_p}$
  from~\eqref{eq:F2}, and by the expression
  of $\TopNorm{\Norm{\cdot}}{l, q}$
  in Proposition~\ref{pr:table_norms}.
  \medskip
  
  This ends the proof.
\end{proof}

\subsection{Proof of Theorem~\ref{th:capra_convexity_and_subdifferentiability}}
\label{subsec:Proof_of_Theorem}

For the proof of Theorem~\ref{th:capra_convexity_and_subdifferentiability},
we use the following result that is essentially an application of~\cite[Proposition~10.1]{Singer:1997}
to the special case of \Capra\ conjugacies.

\begin{fact}
  \label{fact:domain}
  Let $\TripleNorm{\cdot}$ be a norm on $\RR^d$
  and $\CouplingCapra$ be the \Capra\ coupling 
  as in Definition~\ref{de:Capra}, inducing
  the definitions of the \Capra-biconjugate~\eqref{eq:Fenchel-Moreau_biconjugate}
  and the \Capra-subdifferential~\eqref{eq:Capra_subdifferential}.
  We have that
  \begin{equation}
    \SFMbi{ \lzero }{\CouplingCapra}(\primal) \neq \lzero(\primal)
    \implies \partial_{\CouplingCapra} \lzero(\primal) = \emptyset
    \eqsepv \forall \primal \in \PRIMAL \eqfinp
  \end{equation}
\end{fact}

We now turn to the proof of Theorem~\ref{th:capra_convexity_and_subdifferentiability}.
\vspace{\topsep}

\begin{proof}
  The proof is structured as follows: for the source norms $\TripleNorm{\cdot} = \Norm{\cdot}_p$ with $p \in \nc{1, \infty}$,
  \begin{description}
  \item \textit{(i)} first, we give the expression
    of the \Capra-biconjugate $\SFMbi{ \lzero }{\CouplingCapra}$,  
  \item \textit{(ii)} second, we give the expression of 
    the \Capra-subdifferential $\partial_{\CouplingCapra} \lzero\np{\primal}$ at $\primal = 0$,
  \item \textit{(iii)} third, we give the expression of 
    the \Capra-subdifferential $\partial_{\CouplingCapra} \lzero\np{\primal}$ at $\primal \neq 0$,
  \item \textit{(iv)} fourth, we give the domain of the \Capra-subdifferential
    $\partial_{\CouplingCapra} \lzero$.
  \end{description}
  
  \textit{(i)} For $p \in \open{1, \infty}$, 
  the norm $\Norm{\cdot}_p$
  and its dual norm $\Norm{\cdot}_q$ (with $\frac{1}{p} + \frac{1}{q} = 1$)
  are orthant-strictly monotonic,
  following Definition~\ref{de:orthant-monotonic},
  so that $\SFMbi{ \lzero }{\CouplingCapra} = \lzero$,
  from Theorem~\ref{th:capra_convexity}.
  Turning to the case $p \in \na{1, \infty}$,
  we recall that,
  from~\cite[Proposition~4.4]{Chancelier-DeLara:2022_CAPRA_OPTIMIZATION}
  and Proposition~\ref{pr:table_norms}, if $q \in \na{1,\infty}$
  is such that $\frac{1}{p} + \frac{1}{q} = 1$, then we get that
  \begin{equation}
    \lzero^{\CouplingCapra}(\dual)
    = \max_{j=1,\ldots,d} 
    \bp{\TopNorm{\Norm{\dual}}{j, q} - j}^+
    \eqsepv \forall \dual \in \DUAL \eqfinp
    \label{eq:lzero_conjugate}%
  \end{equation}
  
  First, we consider $p = 1$.
  From~\eqref{eq:lzero_conjugate} and~\eqref{eq:top_norm_explicit},
  we have that 
  \begin{align*}
    \lzero^{\CouplingCapra}(\dual)
    &= \max_{j\in\ic{1,d}} 
      \bp{\Norm{\dual}_\infty - j}^+
      = \bp{\Norm{\dual}_\infty - 1}^+
      \eqsepv \forall \dual \in \DUAL \eqfinp
      \intertext{Thus, by definition of the \Capra-biconjugate 
      	in~\eqref{eq:Fenchel-Moreau_biconjugate},
      	we have that $\SFMbi{ \lzero }{\CouplingCapra}(0) = 0$,
      	and that for any \( \primal \in \PRIMAL \setminus \na{0} \),}
    \SFMbi{ \lzero }{\CouplingCapra}(\primal)
    &= \sup_{ \dual \in \DUAL } 
      \Bp{ \frac{\proscal{\primal}{\dual}}{\Norm{\primal}_1} 
      - \bp{\Norm{\dual}_\infty - 1}^+ }
    \\
    &=
      \max
      \Bp{
      \sup_{\Norm{\dual}_\infty \leq 1} \frac{\proscal{\primal}{\dual}}{\Norm{\primal}_1} \eqsepv
      1 + \sup_{\Norm{\dual}_\infty \geq 1} \frac{\proscal{\primal}{\dual}}{\Norm{\primal}_1} - \Norm{\dual}_\infty}
    \\
    &=
      1 
      \eqfinv
  \end{align*}
  since
  $\sup_{\Norm{\dual}_\infty \leq 1}
  \proscal{\primal}{\dual} = \Norm{\primal}_1$,
  by $\Norm{\cdot}_1 = \bp{\Norm{\cdot}_\infty}^*$,
  and $\proscal{\primal}{\dual} \leq \Norm{\primal}_1 \Norm{\dual}_\infty$,
  by H{\"o}lder's inequality.
  This proves~\eqref{eq:biconjugate_1}.
  
  Second, we consider $p = \infty$.
  From~\eqref{eq:lzero_conjugate} and~\eqref{eq:top_norm_explicit},
  for $\dual \in \DUAL$
  and $\nu$ a permutation of $\ic{1, d}$ such that 
  $\abs{\dual_{\nu(1)}} \geq \ldots \geq \abs{\dual_{\nu(d)}}$, we have
  that 
  \begin{align*}
    \lzero^{\CouplingCapra}(\dual)
    &= \max_{j\in\ic{1,d}} 
      \bgp{\Bp{\sum_{k=1}^j \abs{\dual_{\nu(k)}}} - j}^+
      = 
      \sum_{k=1}^d \np{\abs{\dual_{\nu(k)}} - 1} 
      \boldsymbol{1}_{\abs{\dual_{\nu(k)}} \geq 1}
      = \sum_{k=1}^d \np{\abs{\dual_{k}} - 1} 
      \boldsymbol{1}_{\abs{\dual_{k}} \geq 1}
      \eqsepv
      \forall \dual \in \DUAL \eqfinp
      \intertext{Thus, by definition of the \Capra-biconjugate 
      	in~\eqref{eq:Fenchel-Moreau_biconjugate},
      	we have that $\SFMbi{ \lzero }{\CouplingCapra}(0) = 0$,
      	and that for any \( \primal \in \PRIMAL \setminus \na{0} \),} 
    \SFMbi{ \lzero }{\CouplingCapra}(\primal)
    &= \sup_{ \dual \in \DUAL } 
      \Bp{ \frac{\proscal{\primal}{\dual}}{\Norm{\primal}_\infty} 
      - \sum_{k=1}^d \np{\abs{\dual_{k}} - 1} 
      \boldsymbol{1}_{\abs{\dual_{k}} \geq 1} }
    \\
    &= \sum_{k=1}^d
      \sup_{ \dual_k \in \RR } 
      \Bp{ \frac{\primal_k\dual_k}{\Norm{\primal}_\infty} 
      -  \np{\abs{\dual_{k}} - 1} 
      \boldsymbol{1}_{\abs{\dual_{k}} \geq 1} }
    \\
    &= \sum_{k=1}^d \max 
      \Bp{
      \sup_{ \abs{\dual_k} \leq 1}
      \frac{\primal_k\dual_k}{\Norm{\primal}_\infty} \eqsepv
      1 +
      \sup_{ \abs{\dual_k} \geq 1 }
      \frac{\primal_k\dual_k}{\Norm{\primal}_\infty} - \abs{\dual_k}
      }
     \\
    &= \sum_{k=1}^d \frac{\abs{\primal_k}}{\Norm{\primal}_\infty}
      = \frac{\Norm{\primal}_1}{\Norm{\primal}_\infty}
      \eqfinv
  \end{align*}
  since, using similar arguments as above,
  $\sup_{ \abs{\dual_k} \leq 1}\primal_k\dual_k = \abs{\primal_k}$,
  and 
  \[
    1 +
  \sup_{ \abs{\dual_k} \geq 1 }
  \frac{\primal_k\dual_k}{\Norm{\primal}_\infty} - \abs{\dual_k}
  \leq 
  1 +
  \sup_{ \abs{\dual_k} \geq 1 }
  \frac{\abs{\primal_k\dual_k}}{\Norm{\primal}_\infty} - \abs{\dual_k}
  =
  1 +
  \sup_{ \abs{\dual_k} \geq 1 }
  \bp{\frac{\abs{\primal_k}}{\Norm{\primal}_\infty} - 1}\abs{\dual_k}
  = \frac{\abs{\primal_k}}{\Norm{\primal}_\infty}
  \eqfinp
  \]
  This proves~\eqref{eq:biconjugate_infty}. 
  
  \medbreak
  \textit{(ii)} Let us recall that 
  $\partial_{\CouplingCapra} \lzero(0) = 
  \bigcap_{j \in \ic{1,d}} j \BALL_{(j),\star}^{\mathrm{\FlatRR}}$,
  from \eqref{eq:capra_subdifferential_at_0}. 
  If $p=1$,
  from~\eqref{eq:table_norms} we get that 
  $\CoordinateNormDual{\TripleNorm{\cdot}}{j} 
  = \Norm{\cdot}_\infty$, $\forall j \in \ic{1,d}$.
  We deduce that $\partial_{\CouplingCapra} \lzero(0) =
  \BALL_{\Norm{\cdot}_\infty}$.   
  We now assume that $p \in \leftopen{1, \infty}$.
  From~\eqref{eq:table_norms}, we get that 
  $\CoordinateNormDual{\TripleNorm{\cdot}}{j}=
  \TopNorm{\Norm{\cdot}}{j, q}$,
  $\forall j \in \ic{1,d}$
  (with $\frac{1}{p} + \frac{1}{q} = 1$, $q \in \rightopen{1, \infty}$).
  For $j=1$, from~\eqref{eq:top_norm_explicit}, we get that 
  $\TopNorm{\Norm{\cdot}}{1, q} = \Norm{\cdot}_\infty$,
  hence that $\TopNorm{\BALL}{1,q} = \BALL_{\Norm{\cdot}_\infty}$.
  Letting $j > 1$, we prove the inclusion $\BALL_{\Norm{\cdot}_\infty} \subseteq
  j\TopNorm{\BALL}{j,q}$.
  Indeed, we have that 
  \begin{align*}
    \dual \in \BALL_{\Norm{\cdot}_\infty}
    &\implies
      \abs{\dual_{\nu(1)}}^q \leq 1 \eqfinv
      \tag{where $\abs{\dual_{\nu(1)}} = \Norm{\dual}_\infty$}
    \\
    &\implies
      \sum_{i=1}^j \abs{\dual_{\nu(1)}}^q \leq \sum_{i=1}^j 1 = j
      \eqfinv
      \tag{where $\abs{\dual_{\nu(1)}} \geq \ldots \geq \abs{\dual_{\nu(d)}}$}
    \\
    &\implies
      \Bp{\sum_{i=1}^j \abs{\dual_{\nu(1)}}^q}^\frac{1}{q} \leq j^\frac{1}{q} \eqfinv
    \\
    &\implies
      \TopNorm{\Norm{\dual}}{j, q} \leq j
      \tag{by definition of $\TopNorm{\Norm{\cdot}}{j, q}$
      in~\eqref{eq:top_norm_explicit}
      and from $j \geq j^\frac{1}{q}$}
    \\
    &\implies
      \dual \in j\TopNorm{\BALL}{j,q}
      \eqfinp
  \end{align*}
  We conclude that $\partial_{\CouplingCapra} \lzero(0) =
  \bigcap_{j \in \ic{1,d}} \TopNorm{\BALL}{j,q} =
  \BALL_{\Norm{\cdot}_\infty} \cap \Bp{\bigcap_{j \in \ic{2,d}} \TopNorm{\BALL}{j,q} }=
  \BALL_{\Norm{\cdot}_\infty}$.
  
  \medbreak
  \textit{(iii)}
  Let us recall that, 
  for $\primal \neq 0$ and $\lzero(\primal) = \lzerovalue$,
  we have that 
  $\partial_{\CouplingCapra} \lzero(\primal) 
  =
  \NormalCone_{\OriginalSupportNorm{\BALL}{p, \lzerovalue}} \Bp{\frac{\primal}{\OriginalSupportNorm{\Norm{\primal}}{p, \lzerovalue}}}
  \cap
  \Dual_{\lzerovalue}$,  from~\eqref{eq:capra_subdifferential}.
  If $p \in \leftopen{1, \infty}$,
  the expressions of $\partial_{\CouplingCapra} \lzero(\primal)$
  in~\eqref{eq:explicit_capra_subdifferential_lp}
  and in~\eqref{eq:explicit_capra_subdifferential_l_infty}
  are obtained combining Proposition~\ref{pr:normal_cone}
  and Proposition~\ref{pr:admissible_dual_equivalence}.
  If $p=1$,
  for $\lzerovalue \geq 2$, 
  $\Dual_\lzerovalue = \emptyset$ from~\eqref{eq:admissible_dual_l1}
  and thus
  $\partial_{\CouplingCapra} \lzero(\primal) = \emptyset$.
  We now turn to the case $\lzerovalue = 1$,
  denoting $\support = \SupportMapping(\primal) = \na{k}$ where $k \in \ic{1, d}$.
  We have that
  \begin{align*}
    \dual \in \NormalCone_{\OriginalSupportNorm{\BALL}{1, 1}} \Bp{\frac{\primal}{\OriginalSupportNorm{\Norm{\primal}}{1, 1}}}
    &\iff 
      \begin{cases}
        \dual_\support \in \NormalCone_{\BALL_{\Norm{\cdot}_1}}
        \np{\frac{\primal}{\Norm{\primal}_1}} \eqfinv
        \\
        \abs{\dual_j} \leq \min_{i \in \support} \abs{\dual_i} \eqsepv \forall j \notin \support \eqfinv
      \end{cases}
    \tag{from Proposition~\ref{pr:normal_cone}}
    \\
    &\iff
      \begin{cases}
        \Norm{\primal}_1 \Norm{\dual_\support}_\infty = \proscal{\primal}{\dual_\support}  \eqfinv
        \\
        \Norm{\dual_\support}_\infty = \Norm{\dual}_\infty \eqfinv
      \end{cases}
    \tag{from~\eqref{eq:couple_TripleNorm-dual_and_normal_cone} and by definition of $\Norm{\cdot}_\infty$}
    \\
    &\iff
      \Norm{\primal}_1 \Norm{\dual}_\infty = \proscal{\primal}{\dual}  \eqfinv
      \tag{from $\proscal{\primal}{\dual} = \primal_k \dual_k$ and $\Norm{\primal}_1 = \abs{\primal_k}$}
    \\
    &\iff
      \dual \in \NormalCone_{\BALL_{\Norm{\cdot}_1}}
      \np{\frac{\primal}{\Norm{\primal}_1}} \eqfinv
      \tag{from~~\eqref{eq:couple_TripleNorm-dual_and_normal_cone}}
  \end{align*}
  therefore, we deduce from~\eqref{eq:admissible_dual_l1}
  the expression of $\partial_{\CouplingCapra} \lzero(\primal)$
  in~\eqref{eq:explicit_capra_subdifferential_l1}.
  
  \medbreak
  \textit{(iv)} For $p \in \open{1, \infty}$, 
  the norm $\Norm{\cdot}_p$
  and its dual norm $\Norm{\cdot}_q$ (with $\frac{1}{p} + \frac{1}{q} = 1$)
  are orthant-strictly monotonic,
  following Definition~\ref{de:orthant-monotonic},
  so that $\dom\bp{\partial_{\CouplingCapra} \lzero} = \PRIMAL$,
  from Theorem~\ref{th:capra_convexity}.
  We now turn to the case $p \in \na{1, \infty}$.
  
  First, we consider $p = 1$.
  Given the expression of $\SFMbi{ \lzero }{\CouplingCapra}$
  in~\eqref{eq:biconjugate_1}, 
  $\SFMbi{ \lzero }{\CouplingCapra}(\primal) = \lzero(\primal) \iff \lzero(\primal) \leq 1$.
  We deduce from Fact~\ref{fact:domain} that 
  $\dom\bp{\partial_{\CouplingCapra} \lzero} \subseteq \defset{\primal \in \PRIMAL}{\lzero(\primal) \leq 1}$.
  We prove the reciprocal inclusion.
  We already know from~\eqref{eq:explicit_capra_subdifferential_l1}
  that $\partial_{\CouplingCapra} \lzero(0) \neq \emptyset$.
  Let $\primal \in \PRIMAL$ be such that $\lzero(\primal) = 1$. 
  There exists $k \in \ic{1,d}$ such that
  $\SupportMapping(\primal) = \na{k}$.
  Let us introduce $\dual \in \DUAL$ such that $\SupportMapping(\dual) = \na{k}$ with 
  $\dual_k \in \na{-1, 1}$ and $\primal_k \dual_k = \abs{\primal_k}$.
  It follows that $\Norm{\dual}_\infty = 1$ and $\Norm{\primal}_1 \Norm{\dual}_\infty = \primal_k\dual_k = \proscal{\primal}{\dual}$
  and thus that $\dual \in \NormalCone_{\BALL_{\Norm{\cdot}_1}}
  \np{\frac{\primal}{\Norm{\primal}_1}}$, from~\eqref{eq:couple_TripleNorm-dual_and_normal_cone}.
  We deduce from~\eqref{eq:explicit_capra_subdifferential_l1} that $\dual \in \partial_{\CouplingCapra} \lzero(\primal)$,
  hence that $\partial_{\CouplingCapra} \lzero(\primal) \neq \emptyset$. 
  This proves the reciprocal inclusion, and we conclude that
  $\dom\bp{\partial_{\CouplingCapra} \lzero} = \defset{\primal \in \PRIMAL}{\lzero(\primal) \leq 1}$.
  
  Second, we consider $p = \infty$.
  Given the expression of $\SFMbi{ \lzero }{\CouplingCapra}$
  in~\eqref{eq:biconjugate_infty}, we have that $\SFMbi{ \lzero }{\CouplingCapra}(0) = \lzero(0)$,
  and for $\primal \neq 0$,
  \begin{align*}
    \SFMbi{ \lzero }{\CouplingCapra}(\primal) = \lzero(\primal) 
    &\iff 
      \lzerovalue = \sum_{i=1}^l \frac{\abs{\primal_{\nu(i)}}}{\abs{\primal_{\nu(1)}}}
      \eqfinv \tag{where $\lzero(\primal) = \lzerovalue$ and $\abs{\primal_{\nu(1)}} \geq \ldots \geq \abs{\primal_{\nu(d)}}$} 
    \\
    &\iff 
      \begin{cases}
        \abs{\primal_{\nu(k)}} = \abs{\primal_{\nu(1)}} \eqsepv \forall k \in \ic{1,\lzerovalue}
        \eqfinv \\
        \primal_{\nu(k)} = 0 \eqsepv \forall k \in \ic{\lzerovalue+1,d}
        \; (\text{when } l \not=d)
        \eqfinp
      \end{cases}
  \end{align*}
  We deduce that $\SFMbi{ \lzero }{\CouplingCapra}(\primal) = \lzero(\primal)
  \iff \primal \in \cup_{\lambda > 0} \na{-\lambda, 0, \lambda}^d $,
  and thus from Fact~\ref{fact:domain} that 
  $\dom\bp{\partial_{\CouplingCapra} \lzero} \subseteq \cup_{\lambda > 0} \na{-\lambda, 0, \lambda}^d$.
  We prove the reciprocal inclusion.
  We already know from~\eqref{eq:explicit_capra_subdifferential_at_0_l_infty}
  that $\partial_{\CouplingCapra} \lzero(0) \neq \emptyset$.
  Let $\primal \in \cup_{\lambda > 0} \na{-\lambda, 0, \lambda}^d$ be such that $\primal \neq 0$,
  and $\nu$ be a permutation of $\ic{1,d}$ such that $\abs{\primal_{\nu(1)}} \geq \ldots \geq \abs{\primal_{\nu(d)}}$.
  Let us introduce $\dual \in \DUAL$ such that
  $\abs{\dual_{\nu(k)}} = 1 \eqsepv \forall k \in \ic{1,\lzerovalue}$ and
  $\abs{\dual_{\nu(k)}} = 0 \eqsepv \forall k \in \ic{\lzerovalue+1, d}$. 
  It follows that, denoting $\support = \SupportMapping(\primal)$,
  $\Norm{\primal}_\infty\Norm{\dual_{\support}}_1 = \lambda \lzerovalue = \proscal{\primal}{\dual_\support}$,
  and thus that $\dual_{\support} \in \NormalCone_{\BALL_{\Norm{\cdot}_\infty}}
  \np{\frac{\primal}{\Norm{\primal}_\infty}}$, from~\eqref{eq:couple_TripleNorm-dual_and_normal_cone}.
  We deduce from~\eqref{eq:explicit_capra_subdifferential_l_infty} that $\dual \in \partial_{\CouplingCapra} \lzero(\primal)$,
  hence that $\partial_{\CouplingCapra} \lzero(\primal) \neq \emptyset$. 
  This proves the reciprocal inclusion, and we conclude that
  $\dom\bp{\partial_{\CouplingCapra} \lzero} = \cup_{\lambda > 0} \na{-\lambda,
    0, \lambda}^d$.
  \medskip

  This ends the proof. 
\end{proof}

\section{Graphical representations and discussion}
\label{sec:visualization_and_discussion}%

First, we provide graphical representations of 
the \Capra-subdifferential of the $\lzero$ pseudonorm 
in \S\ref{sec:numerical_example}.
Second, we compare our expression of
$\partial_{\CouplingCapra}\lzero$ 
with other notions of generalized subdifferential 
for the $\lzero$ pseudonorm and illustrate one of its applications in \S\ref{sec:capra_subdifferential_discussion}.

\subsection{Visualization with the $\ell_2$ source norm}
\label{sec:numerical_example}

We detail the \Capra-subdifferential of~$\lzero$ for the $\ell_2$ source norm
$\TripleNorm{\cdot} = \Norm{\cdot}_2$.
According to Theorem~\ref{th:capra_convexity_and_subdifferentiability},
we have that
\begin{subequations}
  \label{eq:capra_l2}%
  \begin{align}
    \partial_{\CouplingCapra} \lzero(0) 
    &= 
      \BALL_{\Norm{\cdot}_\infty}
      \eqfinv 
      \intertext{and for $\primal \neq 0$, $\dual \in \DUAL$,
      denoting $\lzerovalue = \lzero(\primal)$,
      $\support = \SupportMapping(\primal)$, and
      $\nu$ a permutation of $\ic{1, d}$ such that 
      $\abs{\dual_{\nu(1)}} \geq \ldots \geq \abs{\dual_{\nu(d)}}$,}
      \dual \in
      \partial_{\CouplingCapra} \lzero(\primal) 
    &\iff
      \begin{cases}
	\dual_\support = \lambda\primal \eqsepv
	\lambda \geq 0
	\eqfinv \\	
	\abs{\dual_j} \leq \min_{i \in \support} \abs{\dual_i} \eqsepv 
	\forall j \notin \support \eqfinv \\
	\abs{\dual_{\nu(k+1)}}^2 \geq \bp{\TopNorm{\Norm{\dual}}{k, 2} + 1}^2 
	- \bp{\TopNorm{\Norm{\dual}}{k, 2}}^2 \eqsepv 
	\forall k \in \ic{0, \lzerovalue-1} 
	\eqfinv
	\\
	\abs{\dual_{\nu(\lzerovalue+1)}}^2
	\leq 
	\bp{\TopNorm{\Norm{\dual}}{\lzerovalue, 2} + 1}^2 
	- \bp{\TopNorm{\Norm{\dual}}{\lzerovalue, 2}}^2
	\eqfinp
      \end{cases}
  \end{align}
\end{subequations}

We illustrate in Figure~\ref{fig:capra_subdifferential_2D}
the \Capra-subdifferentials obtained with~\eqref{eq:capra_l2}
in the two-dimensional case where $\lzero : \RR^2 \to \na{0,1,2}$.
In Figure~\ref{fig:capra_subdiff_on_points}, we display the
\Capra-subdifferential of $\lzero$ at three typical points,
covering the three possible cases in $\RR^2$,
with $\lzero(\primal) = 0$ (green color),
$\lzero(\primal) = 1$ (red color), and
$\lzero(\primal) = 2$ (blue color).
Then, using the same colors, 
we display in Figure~\ref{fig:capra_subdiff_on_shphere}
the \Capra-subdifferential of $\lzero$ at
all points in $\RR^2$.

\pgfmathsetmacro{\lw}{0.6} 

\begin{figure}[htpb]
  
  \begin{subfigure}[b]{\linewidth}
    \centering
    \resizebox{8cm}{!}{
      \begin{tikzpicture}
        \begin{axis}[ 
          ticks=none,
          axis lines = middle,
          axis line style={->},
          ymin=-12, ymax=12,
          xmin=-12, xmax=12,
          xlabel={$\primal_1$},
          ylabel={$\primal_2$},
          x label style={at={(axis description cs:0.95,0.42)},anchor=center},
          y label style={at={(axis description cs:0.42,0.95)},anchor=center},
          axis equal image
          ]
          
          \addplot[name path=inftyR, mygreen, line width=\lw pt] coordinates {(1,-1) (1,1)};
          \addplot[name path=inftyL, mygreen, line width=\lw pt] coordinates {(-1,-1) (-1,1)};
          \addplot[name path=inftyT, mygreen, line width=\lw pt] coordinates {(-1,1) (1,1)};
          \addplot[name path=inftyB, mygreen, line width=\lw pt] coordinates {(-1,-1) (1,-1)};
          \addplot[color=mygreen, opacity=0.5]fill between[of=inftyB and inftyT, soft clip={domain=-1:1}];
          \addplot[black, mark=*, only marks, mark size=1.5pt] coordinates {(0,0)};
          
          \addplot[red, mark=] coordinates {(1,-1) (1,1)};
          \addplot[name path=linear+, red, domain=1:1+sqrt(2)] {x};
          \addplot[name path=linear-, red, domain=1:1+sqrt(2)] {-x};
          \addplot[color=red!30, opacity=0.5]fill between[of=linear+ and linear-, soft clip={domain=1:2.41}];
          \addplot[name path=parabolic+, red, domain=1+sqrt(2):12] {sqrt(2*x+1)};
          \addplot[name path=parabolic-, red, domain=1+sqrt(2):12] {-sqrt(2*x+1)};
          \addplot[color=red!30, opacity=0.5]fill between[of=parabolic+ and parabolic-, soft clip={domain=2.4:12}];
          \addplot[black, mark=*, only marks, mark size=1.5pt] coordinates {(1,0)};
          
          
          \addplot[black, mark=*, only marks, mark size=1.5pt] coordinates {(-{sqrt(3)/2},-0.5)};
          \addplot[black!30!blue, mark=, thick] coordinates {({(1+cos(sqrt(3)/2))*4*(-sqrt(3)/2)},{(1+cos(sqrt(3)/2))*4*(-0.5)}) (-30*(sqrt(3)/2),-30*0.5)};
          
          \draw (axis cs:0,0) circle [black, radius=1];
        \end{axis}
      \end{tikzpicture}}
    \caption{$\textcolor{mygreen}{\partial_{\CouplingCapra} \lzero \np{0, 0}} \eqsepv
      \rouge{\partial_{\CouplingCapra} \lzero \np{1, 0}} \eqsepv
      \textcolor{myblue}{\partial_{\CouplingCapra} \lzero
        \np{-\frac{\sqrt{3}}{2}, -\frac{1}{2}}}$} 
    \label{fig:capra_subdiff_on_points}
  \end{subfigure}
  
  \hfill
  
  \begin{subfigure}[b]{\linewidth}
    \centering
    \resizebox{8cm}{!}{
      \begin{tikzpicture}
        \begin{axis}[ 
          ticks=none,
          axis lines = middle,
          axis line style={->},
          ymin=-12, ymax=12,
          xmin=-12, xmax=12,
          xlabel={$\primal_1$},
          ylabel={$\primal_2$},
          x label style={at={(axis description cs:0.95,0.42)},anchor=center},
          y label style={at={(axis description cs:0.42,0.95)},anchor=center},
          axis equal image,
          width=8cm,
          height=8cm
          ]
          
          \addplot[red, mark=, line width=\lw pt] coordinates {(1,-1) (1,1)};
          \addplot[name path=linear+, opacity=0., domain=1:1+sqrt(2)] {x};
          \addplot[name path=linear-, opacity=0., domain=1:1+sqrt(2)] {-x};
          \addplot[color=red!30, opacity=0.5]fill between[of=linear+ and linear-, soft clip={domain=1:2.41}];
          \addplot[name path=parabolic+, red, domain=1+sqrt(2):12,  line width=\lw pt] {sqrt(2*x+1)};
          \addplot[name path=parabolic-, red, domain=1+sqrt(2):12, line width=\lw pt] {-sqrt(2*x+1)};
          \addplot[color=red!30, opacity=0.5]fill between[of=parabolic+ and parabolic-, soft clip={domain=2.4:12}];
          \addplot[name path=polar+, blue, dash pattern=on 2pt off 2pt, dash phase=2pt, data cs=polar, domain=5:45, line width=\lw pt] {(1+cos(x)) / (sin(x)*sin(x))};
          \addplot[name path=bisectrice+, opacity=0., domain=1+sqrt(2):12] {x};
          \addplot[color=blue!30, opacity=0.5]fill between[of=bisectrice+ and parabolic+, soft clip={domain=2.4:12}];
          \addplot[name path=polar1, blue, dash pattern=on 2pt off 2pt, data cs=polar, domain=315:355, line width=\lw pt] {(1+cos(x)) / (sin(x)*sin(x))};
          \addplot[name path=bisectrice-, opacity=0., domain=1+sqrt(2):12] {-x};
          \addplot[color=blue!30, opacity=0.5]fill between[of=bisectrice- and parabolic-, soft clip={domain=2.4:12}];
          \addplot[name path=inftyR, mygreen, dash pattern=on 2pt off 2pt, dash phase=2pt, line width=\lw pt] coordinates {(1,-1) (1,1)};
          
          
          \addplot[red, mark=, line width=\lw pt] coordinates {(-1,-1) (-1,1)};
          \addplot[name path=linear+, opacity=0., domain=-1-sqrt(2):-1] {x};
          \addplot[name path=linear-, opacity=0., domain=-1-sqrt(2):-1] {-x};
          \addplot[color=red!30, opacity=0.5]fill between[of=linear+ and linear-, soft clip={domain=-1:-2.41}];
          \addplot[name path=parabolic+, red, domain=-1-sqrt(2):-12, line width=\lw pt] {sqrt(-2*x+1)};
          \addplot[name path=parabolic-, red, domain=-1-sqrt(2):-12, line width=\lw pt] {-sqrt(-2*x+1)};
          \addplot[color=red!30, opacity=0.5]fill between[of=parabolic+ and parabolic-, soft clip={domain=-2.4:-12}];
          \addplot[name path=polar+, blue, dash pattern=on 2pt off 2pt, dash phase=2pt, data cs=polar, domain=135:180, line width=\lw pt] {(1-cos(x)) / (sin(x)*sin(x))};
          \addplot[name path=bisectrice+, opacity=0., domain=-1-sqrt(2):-12] {-x};
          \addplot[color=blue!30, opacity=0.5]fill between[of=bisectrice+ and parabolic+, soft clip={domain=-2.4:-12}];
          \addplot[name path=polar1, blue, dash pattern=on 2pt off 2pt, data cs=polar, domain=180:225, line width=\lw pt] {(1-cos(x)) / (sin(x)*sin(x))};
          \addplot[name path=bisectrice-, opacity=0., domain=-1-sqrt(2):-12] {x};
          \addplot[color=blue!30, opacity=0.5]fill between[of=bisectrice- and parabolic-, soft clip={domain=-2.4:-12}];
          \addplot[name path=inftyL, mygreen, dash pattern=on 2pt off 2pt, dash phase=2pt, line width=\lw pt] coordinates {(-1,-1) (-1,1)};
          
          
          \pgfmathsetmacro\xx{1+sqrt(2.)}
          
          \addplot[red, mark=, line width=\lw pt] coordinates {(-1,1) (1,1)};
          \addplot[name path=cone+, opacity=0., domain=-12:12] {max(-x,x,1)};
          \addplot[name path=parabolic+, red, domain=\xx:12, line width=\lw pt] {(x^2-1)/2};
          \addplot[blue, dash pattern=on 2pt off 2pt, dash phase=2pt, domain=\xx:12, line width=\lw pt] {(x^2-1)/2};
          \addplot[name path=parabolic-, red, domain=-\xx:-12, line width=\lw pt] {(x^2-1)/2};
          \addplot[blue, dash pattern=on 2pt off 2pt, dash phase=2pt, domain=-\xx:-12, line width=\lw pt] {(x^2-1)/2};
          \addplot[name path=top, opacity=0., domain=-12:12] {12};
          \addplot[color=red!30, opacity=0.5]fill between[of=cone+ and top, soft clip={domain=-\xx:\xx}];
          \addplot[color=red!30, opacity=0.5]fill between[of=parabolic+ and top, soft clip={domain=\xx:12}];
          \addplot[color=red!30, opacity=0.5]fill between[of=parabolic- and top, soft clip={domain=-\xx:-12}];
          \addplot[color=blue!30, opacity=0.5]fill between[of=parabolic+ and cone+, soft clip={domain=\xx:12}];
          \addplot[color=blue!30, opacity=0.5]fill between[of=parabolic- and cone+, soft clip={domain=-\xx:-12}];
          \addplot[name path=inftyT, mygreen, dash pattern=on 2pt off 2pt, dash phase=2pt, line width=\lw pt] coordinates {(-1,1) (1,1)};
          
          \addplot[red, mark=, line width=\lw pt] coordinates {(-1,-1) (1,-1)};
          \addplot[name path=cone-, opacity=0., domain=-12:12] {min(-x,x,-1)};
          \addplot[name path=parabolic+, red, domain=1+sqrt(2):12, line width=\lw pt] {-(x^2-1)/2};
          \addplot[blue, dash pattern=on 2pt off 2pt, dash phase=2pt, domain=1+sqrt(2):12, line width=\lw pt] {-(x^2-1)/2};
          \addplot[name path=parabolic-, red, domain=-1-sqrt(2):-12, line width=\lw pt] {-(x^2-1)/2};
          \addplot[blue, dash pattern=on 2pt off 2pt, dash phase=2pt, domain=-1-sqrt(2):-12, line width=\lw pt] {-(x^2-1)/2};
          \addplot[name path=bottom, opacity=0., domain=-12:12] {-12};
          \addplot[color=red!30, opacity=0.5]fill between[of=cone- and bottom, soft clip={domain=-\xx:\xx}];
          \addplot[color=red!30, opacity=0.5]fill between[of=parabolic+ and bottom, soft clip={domain=\xx:12}];
          \addplot[color=red!30, opacity=0.5]fill between[of=parabolic- and bottom, soft clip={domain=-\xx:-12}];
          \addplot[color=blue!30, opacity=0.5]fill between[of=parabolic+ and cone-, soft clip={domain=\xx:12}];
          \addplot[color=blue!30, opacity=0.5]fill between[of=parabolic- and cone-, soft clip={domain=-\xx:-12}];
          \addplot[name path=inftyB, mygreen, dash pattern=on 2pt off 2pt, dash phase=2pt, line width=\lw pt] coordinates {(-1,-1) (1,-1)};
          
          \addplot[color=mygreen, opacity=0.5]fill between[of=inftyB and inftyT, soft clip={domain=-1:1}];
          
        \end{axis}
      \end{tikzpicture}}
    \caption{$\textcolor{mygreen}{\partial_{\CouplingCapra} \lzero (0)}
      \bigcup
      \rouge{\Ba{\underset{\lzero(\primal)=1}{\bigcup}
          \partial_{\CouplingCapra} \lzero (\primal)}}
      \bigcup
      \bleu{\Ba{\underset{\lzero(\primal)=2}{\bigcup}
          \partial_{\CouplingCapra} \lzero (\primal)}}$} 
    \label{fig:capra_subdiff_on_shphere}
  \end{subfigure}
  \caption{\Capra-subdifferential of the $\lzero$ pseudonorm
    in $\RR^2$ with the $\ell_2$ source norm
    $\TripleNorm{\cdot} = \Norm{\cdot}_2$,
    illustrated for three typical points (Figure~\ref{fig:capra_subdiff_on_points})
    and for all points in $\RR^2$ (Figure~\ref{fig:capra_subdiff_on_shphere})}
  \label{fig:capra_subdifferential_2D}
\end{figure}

\subsection{Discussion}
\label{sec:capra_subdifferential_discussion}

First,
we compare the \Capra-subdifferential
of the $\lzero$ pseudonorm given
in Theorem~\ref{th:capra_convexity_and_subdifferentiability}
with other notions of subdifferentials.
We recall that, for $\lzero$, the standard subdifferential
of convex analysis obtained with the Fenchel conjugacy
is given by (see \cite[Table 3]{Chancelier-DeLara:2022_CAPRA_OPTIMIZATION})
\begin{equation}
  \partial \lzero(0) = \na{0}
  \text{ and } \enspace
  \partial \lzero(\primal) = \emptyset \eqsepv
  \forall \primal \in \PRIMAL \setminus \na{0}
  \eqfinp
  \label{eq:subdifferential_l0}%
\end{equation}
We also recall further notions of generalized subdifferentials
obtained for the $\lzero$ pseudonorm.
We refer the reader to \cite{Le:2013}
for the definitions of the Fr{\'e}chet, viscosity, proximal, Clarke
and limiting subdifferentials, where the author establishes
that all these notions coincide for the $\lzero$ pseudonorm,
and are equal to 
the set-valued mapping
\begin{equation}
  \mathcal{D} : \PRIMAL \rightrightarrows \DUAL \eqsepv 
  \primal \mapsto 
  \defset{\dual \in \DUAL}{\dual_\support = 0}
  \eqfinv
  \label{eq:generalized_subdiff}%
\end{equation}
where $\support = \SupportMapping(\primal)$,
from \cite[Theorems 1, 2]{Le:2013}.
We deduce that
the \Capra-subdifferential
of the $\lzero$ pseudonorm is significantly different
from previous notions of generalized subdifferentials of $\lzero$,
summarized by~\eqref{eq:subdifferential_l0} and~\eqref{eq:generalized_subdiff}.
In particular, whereas \( \defset{\dual \in \DUAL}{\dual_\support = 0} \) is a
vector subspace, the \Capra-subdifferential
\( \partial_{\CouplingCapra} \lzero(\primal) \) is a closed convex set,
but not a vector subspace.
However, we recall that the \Capra-subdifferential of $\lzero$
is related to the standard subdifferential of $\Lzero$ --- 
a proper closed convex function first introduced in~\cite[\S4.1]{Chancelier-DeLara:2021_ECAPRA_JCA}
for the Euclidean source norm,
then generalized in \cite[Equation~(19), Proposition 3]{Chancelier-DeLara:2021_SVVA} --- 
that ``factorizes'' the $\lzero$ pseudonorm, in the sense that $\lzero = \Lzero
\circ \normalized$,
where $\normalized : \PRIMAL \to \SPHERE_{\TripleNorm{\cdot}}\cup\na{0}$
is the normalization mapping such that $\CouplingCapra\np{\cdot, \cdot} = \proscal{\normalized(\cdot)}{\cdot}$
in~\eqref{eq:coupling_CAPRA}.
Indeed, by application of~\cite[Item~(c), Proposition~3]{Chancelier-DeLara:2021_SVVA}, 
when the source norm $\TripleNorm{\cdot}$ is a $\ell_p$ norm
with $p \in \open{1, \infty}$,
the \Capra-subdifferential of $\lzero$
and the standard subdifferential of $\Lzero$
coincide on the unit sphere, that is,
\begin{equation}
  p \in \open{1, \infty} \text{ and } 
  \Norm{\primal}_p=1 \implies 
  \partial_{\CouplingCapra} \lzero(\primal)
  = \partial \Lzero(\primal)
  \eqfinp
  \label{eq:subdifferential_of_L0}%
\end{equation}
It follows that, for these $\ell_p$ source norms,
  Equation~\eqref{eq:explicit_capra_subdifferential_lp} ---
which provides explicit formulas for~$\partial_{\CouplingCapra} \lzero(\primal)$
--- then also gives,  on the unit
sphere~$\SPHERE_{\Norm{\cdot}_p}$,
explicit formulas for the standard subdifferential of the proper closed convex function~$\Lzero$. 

Second, 
we argue that, since the $\lzero$ pseudonorm 
displays the \Capra-convex properties stated
in Theorem~\ref{th:capra_convexity_and_subdifferentiability},
the \Capra-subdifferential is relevant to obtain
lower approximations of the $\lzero$ pseudonorm. 
We recall that nonconvex continuous approximations of the $\lzero$ pseudonorm
have gained a lot of interest in the field of 
sparse optimization, especially due to applications in machine learning
\cite{Soubies:2015, Qu:2016, Wen:2018}.
The lower approximation of $\lzero$ that we propose next
can be seen as a generalization of polyhedral
lower approximations obtained for a proper, lower semicontinuous and convex
function: here, the maximum of a finite number of affine functions 
now translates into ``polyhedral-like''~\cite[p.~114]{Singer:1997} functions
that are the maximum of a finite number of so-called \emph{\Capra-affine} functions, that is,
functions of the form \( \primal \mapsto \CouplingCapra(\primal, \dual) - z \)
for fixed \( \dual\in\RR^d \) and $z \in \barRR$. 

Let the source norm $\TripleNorm{\cdot}$ be a $\ell_p$ norm,
with $p \in \open{1, \infty}$,
and let
$\na{\primal_i}_{i\in I}$ and $\na{\dual_i}_{i\in I}$ be two collections of points 
such that for $i\in I$,
$\primal_i \in \PRIMAL$ 
and $\dual_i \in \partial_{\CouplingCapra} \lzero(\primal_i)$.
By definition of the \Capra-biconjugate in~\eqref{eq:Fenchel-Moreau_biconjugate},
we have that
\begin{equation}
  \max_{i \in I} \Bp{\CouplingCapra(\primal, \dual_i) - \SFM{ \lzero }{\CouplingCapra}(\dual_i)}
  \leq
  \sup_{\dual \in \DUAL} \Bp{ \CouplingCapra\np{\primal,\dual} 
    -\SFM{ \lzero }{\CouplingCapra}(\dual)  } 
  =
  \SFMbi{ \lzero }{\CouplingCapra}(\primal)
  \eqsepv \forall \primal \in \PRIMAL
  \eqfinp
  \label{eq:polyhedral_inequality}%
\end{equation}
Therefore, we deduce from~\eqref{eq:biconjugate_inequality} that the function
\begin{equation}
  \underline{\lzero} : \PRIMAL \to \RR \eqsepv
  \primal \mapsto \max_{i \in I} \Bp{\CouplingCapra(\primal, \dual_i) - \SFM{ \lzero }{\CouplingCapra}(\dual_i)}
  \label{eq:l0_lower}%
\end{equation}
gives a lower bound for~$\lzero$.
Moreover, by definition of the \Capra-subdifferential in~\eqref{eq:Capra_subdifferential},
we have that for, $i\in I$, 
\begin{equation}
  \CouplingCapra(\primal_i, \dual_i) - \SFM{ \lzero }{\CouplingCapra}(\dual_i) = \lzero(\primal_i)
  \eqfinv
\end{equation}
so that
this lower bound is exact (tight) at the points in $\na{\primal_i}_{i\in I}$,
in the sense that $\underline{\lzero}(\primal_i) = \lzero(\primal_i)$.
Thus, we can tighten the inequality in~\eqref{eq:polyhedral_inequality}
by enlarging the collections $\na{\primal_i}_{i\in I}$ and $\na{\dual_i}_{i\in I}$.
We provide an example of such a lower approximation of~$\lzero$
in Figure~\ref{fig:l0_lower}, using the $\ell_2$ source norm $\TripleNorm{\cdot} = \Norm{\cdot}_2$. 
By definition of $\underline{\lzero}$
in~\eqref{eq:l0_lower} and of the \Capra\ coupling in~\eqref{eq:Capra_conjugacies},
it is straightforward to see that $\underline{\lzero}$ is constant along rays,
so that we only give its representation on $\SPHERE\cup\na{0}$ (orange color).
Observe that, at the sample points $\na{\primal_i}_{i\in I}$ (black dots),
$\underline{\lzero}$ takes the same values as $\lzero$ (blue color, Figure~\ref{fig:l0}).

\begin{figure}
  \begin{subfigure}[b]{0.45\linewidth}
    \includegraphics[width=\linewidth]{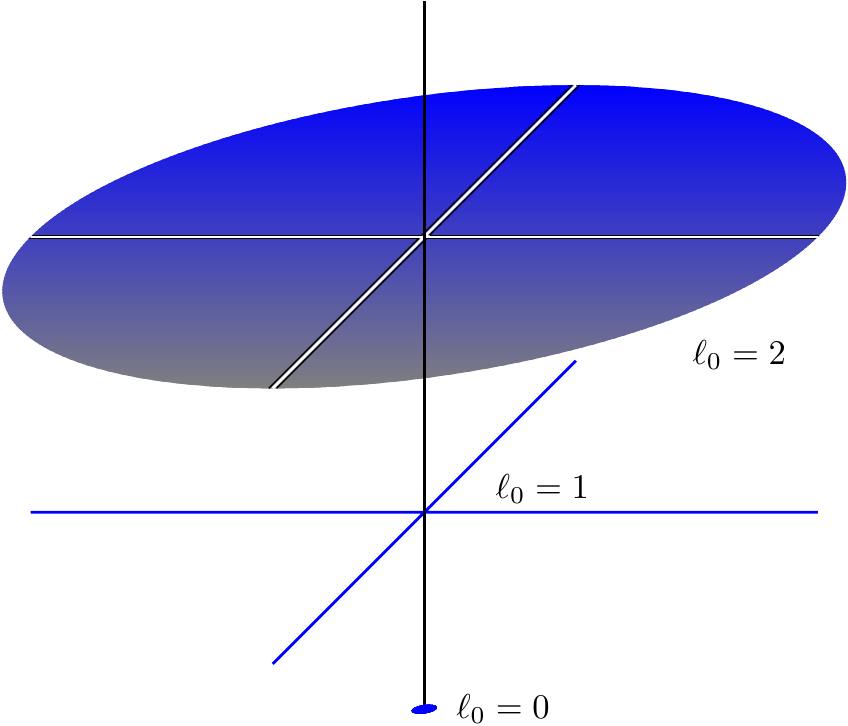}
    \caption{$\lzero : \RR^2 \to \na{0,1,2}$}
    \label{fig:l0}
  \end{subfigure}
  \hfill
  \begin{subfigure}[b]{0.45\linewidth}
    \includegraphics[width=\linewidth]{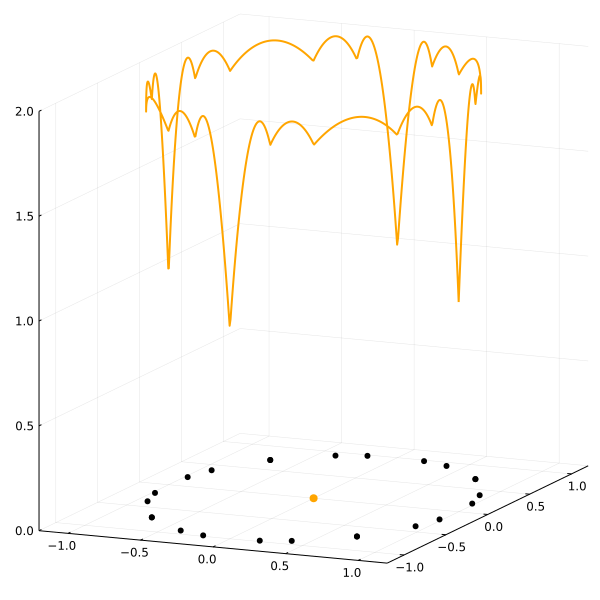}
    \caption{$\underline{\lzero} : \RR^2 \to \RR$ in~\eqref{eq:l0_lower} on $\SPHERE\cup\na{0}$}
    \label{fig:l0_lower}
  \end{subfigure}
  \caption{The $\lzero$ pseudonorm in $\RR^2$ (blue color, Figure~\ref{fig:l0})
    and a ```polyhedral-like''~\cite[p.~114]{Singer:1997} lower bound~$\underline{\lzero}$
    as in~\eqref{eq:l0_lower} represented on $\SPHERE\cup\na{0}$
    (orange color, Figure~\ref{fig:l0_lower})
    obtained for the $\ell_2$ source norm $\TripleNorm{\cdot} = \Norm{\cdot}_2$
    with points
    $\na{\primal_i}_{i\in I}$ sampled on $\SPHERE$  
    (black dots, Figure~\ref{fig:l0_lower})
  }
\end{figure}


\section{Conclusion}

We have derived explicit formulations for the \Capra-subdifferential
of the $\lzero$ pseudonorm for the $\ell_p$ source norms 
with $p \in \nc{1, \infty}$. 
With these formulations, it is now possible to compute 
elements in such \Capra-subdifferentials, 
that we have illustrated by a graphical representation.
On top of that, we have extended previous 
knowledge on $\lzero$, establishing that it is neither 
\Capra-convex nor \Capra-subdifferentiable everywhere 
in the limit cases where $p \in \na{1, \infty}$.

The formulation that we obtain differs drastically from previous
notions of generalized subdifferential for the $\lzero$ pseudonorm.
Whereas most other notions coincide, 
the \Capra-subdifferential enriches this collection
and is an interesting tool to deal with the function $\lzero$, 
in the spirit of the usual notion of subdifferential
for proper lower semicontinuous convex functions.

\paragraph*{Acknowledgement.}{We thank the two anonymous referees that helped us to improve the quality of this paper.}

\appendix

\section{Background on norms}
\label{app:background_on_norms}

For any norm~$\TripleNorm{\cdot}$ on~$\RR^d$,
we introduce derived norms and some of their properties.

\subsubsection*{Dual norms and normal cones}

The following expression 
\begin{equation}
  \TripleNorm{\dual}_\star = 
  \sup_{ \TripleNorm{\primal} \leq 1 } \proscal{\primal}{\dual} 
  \eqsepv \forall \dual \in \RR^d
  \label{eq:dual_norm}
\end{equation}
defines a norm on~$\RR^d$, 
called the \emph{dual norm} \( \TripleNormDual{\cdot} \).
In line with our notations for the norm~$\TripleNorm{\cdot}$ in~\eqref{eq:triplenorm_unit_sphere_ball},
we denote the unit sphere and the unit ball 
of the dual norm~$\TripleNormDual{\cdot}$ by 
\begin{subequations}
  \begin{align}
    \TripleNormDualSphere
    &  = 
      \defset{\dual \in \RR^d}{\TripleNormDual{\dual} = 1} 
      \eqfinv
      \label{eq:triplenorm_Dual_unit_sphere}
    \\
    \TripleNormDualBall
    &  = 
      \defset{\dual \in \RR^d}{\TripleNormDual{\dual} \leq 1} 
      \eqfinp
      \label{eq:triplenorm_Dual_unit_ball}
  \end{align}
\end{subequations}
Note that by definition of the dual norm in~\eqref{eq:dual_norm}, we have 
the inequality
\begin{equation}
  \proscal{\primal}{\dual} \leq 
  \TripleNorm{\primal} \times \TripleNormDual{\dual} 
  \eqsepv \forall  \np{\primal,\dual} \in \RR^d \times \RR^d 
  \eqfinp 
  \label{eq:norm_dual_norm_inequality}
\end{equation}

Equality cases in the above inequality can be characterized in term
of geometric objects of convex analysis. 
For this purpose,
we recall that the \emph{normal cone}~$\NormalCone_{\Convex}(\primal)$ 
to the nonempty closed convex subset~${\Convex} \subseteq \RR^d $
at~$\primal \in \Convex$ is the closed convex cone defined by 
\cite[Definition~5.2.3]{Hiriart-Urruty-Lemarechal:2004}
\begin{equation}
  \NormalCone_{\Convex}(\primal) =
  \bset{ \dual \in \RR^d}
  {
    \proscal{\primal'-\primal}{\dual} \leq 0 \eqsepv 
    \forall \primal' \in \Convex }
  \eqfinp
  \label{eq:normal_cone}
\end{equation}
Now, easy computations show that for any
$\np{\primal,\dual} 
\in \RR^d\backslash\{0\} \times \RR^d\backslash\{0\}$,
we have the equivalence
\begin{equation}
  \proscal{\primal}{\dual} =
  \TripleNorm{\primal} \times \TripleNormDual{\dual} 
  \iff
  \dual \in \NormalCone_{ \TripleNormBall }\bp{ \frac{ \primal }{\TripleNorm{\primal} } }
  \iff
  \primal \in \NormalCone_{ \TripleNormDualBall }\bp{ \frac{ \dual }{
      \TripleNorm{\dual} } } 
  \eqfinp 
  \label{eq:couple_TripleNorm-dual_and_normal_cone} 
\end{equation}

\subsubsection*{Orthant strict monotonicity}

For any \( \primal \in \RR^d \), we denote by \( \module{\primal} \)
the vector of~$ \RR^{d}$ with components $|\primal_i|$, $i=1,\ldots,d$.

\begin{definition}[from \cite{Chancelier-DeLara:2021_SVVA}, Definition~5] 
  A norm \( \TripleNorm{\cdot}\) on the space~\( \RR^d \) is called
  \emph{orthant-strictly monotonic} if, for all 
  $\primal$, $\primal'$ in~$ \RR^{d}$, we have
  \begin{equation}
    \bp{|\primal| < |\primal'| \text{ and } 
      \primal~\circ~\primal' \ge 0} 
    \implies 
    \TripleNorm{\primal} < \TripleNorm{\primal'}
    \eqfinv
  \end{equation} 	
  where \( |\primal| < |\primal'| \) means that 
  $|\primal_i| \le |\primal^{'}_i|$ for all $i=1,\ldots,d$, 
  and that there exists $j \in \na{1,\ldots,d}$, such that
  $|\primal_j| < |\primal^{'}_j|$;
  and $\primal~\circ~\primal' =
  \np{ \primal_1 \primal'_1,\ldots, \primal_d \primal'_d}$
  is the Hadamard (entrywise) product.
  \label{de:orthant-monotonic}%
\end{definition}

\subsubsection*{Restriction norms, coordinate-$k$ and dual coordinate-$k$ norms}

We start by introducing restriction norms and their dual.

\begin{definition}[\cite{Chancelier-DeLara:2022_CAPRA_OPTIMIZATION}, Definition~3.1]   \label{de:K_norm}
  For any norm~$\TripleNorm{\cdot}$ on~$\RR^d$
  and any subset \( K \subseteq \na{1,\ldots,d} \),
  we define two norms on the subspace~\( \FlatRR_{K} \) of~\( \RR^d \),
  as defined in~\eqref{eq:FlatRR}, as follows.
  \begin{itemize}
  \item 
    The \emph{$K$-restriction norm} \( \TripleNorm{\cdot}_{K} \)
    is defined by   
    \begin{equation}
      \TripleNorm{\primal}_{K} = \TripleNorm{\primal} 
      \eqsepv
      \forall \primal \in \FlatRR_{K} 
      \eqfinp 
      \label{eq:K_norm}
    \end{equation}
  \item 
    The $\SetStar{K}$-norm
    \( \TripleNorm{\cdot}_{K,\star} \) is 
    the norm \( \bp{\TripleNorm{\cdot}_{K}}_{\star} \),
    given by the dual norm (on the subspace~\( \FlatRR_{K} \))
    of the restriction norm~\( \TripleNorm{\cdot}_{K} \) 
    to the subspace~\( \FlatRR_{K} \) (first restriction, then dual).
  \end{itemize}
\end{definition}

With these norms, we define the coordinate-$k$ and dual coordinate-$k$ norms.

\begin{definition}[\cite{Chancelier-DeLara:2022_CAPRA_OPTIMIZATION}, Definition~3.2]
  For \( k \in \na{1,\ldots,d} \), we call \emph{coordinate-$k$ norm}
  the norm \( \CoordinateNorm{\TripleNorm{\cdot}}{k} \) 
  whose dual norm is  the 
  \emph{dual coordinate-$k$ norm}, denoted by
  \( \CoordinateNormDual{\TripleNorm{\cdot}}{k} \), 
  with expression
  \begin{equation}
    \CoordinateNormDual{\TripleNorm{\dual}}{k}
    =
    \sup_{\cardinal{K} \leq k} \TripleNorm{\dual_K}_{K,\star} 
    \eqsepv \forall \dual \in \RR^d 
    \eqfinv
    \label{eq:dual_coordinate_norm_definition}
  \end{equation}
  where the $\SetStar{K}$-norm \( \TripleNorm{\cdot}_{K,\star} \) is given in
  Definition~\ref{de:K_norm}, 
  and where the notation \( \sup_{\cardinal{K} \leq k} \) is a shorthand for 
  \( \sup_{ { K \subseteq \na{1,\ldots,d}, \cardinal{K} \leq k}} \).
  \label{de:coordinate_norm}
\end{definition}

Also, following~\cite[\S3.2]{Chancelier-DeLara:2022_CAPRA_OPTIMIZATION}, 
we extend the dual coordinate-$k$
norms in Definition~\ref{de:coordinate_norm}
with the convention
$\CoordinateNormDual{\TripleNorm{\cdot}}{0} = 0$,
although this is not a norm on $\RR^d$ but a seminorm.

\newcommand{\noopsort}[1]{} \ifx\undefined\allcaps\def\allcaps#1{#1}\fi

\end{document}